\documentclass[a4paper,10pt,reqno]{amsart}
\usepackage[pagebackref=true]{hyperref}
\hypersetup{
    bookmarks=true,         
    unicode=false,          
    pdftoolbar=false,        
    pdfmenubar=true,        
    pdftitle={Control of underwater vehicles},    
    pdfauthor={Lecaros,  Rosier},     
    colorlinks=true,       
    linkcolor=blue, 
    citecolor=red, 
    filecolor=magenta,      
    urlcolor=cyan          
}

\usepackage{a4wide,amsmath,amssymb,latexsym,amsthm}
\usepackage{eucal}
\usepackage{graphicx}

\newtheorem{theorem}{Theorem}[section]
\newtheorem{proposition}[theorem]{Proposition}
\newtheorem{remark}[theorem]{Remark}
\newtheorem{lemma}[theorem]{Lemma}

\numberwithin{equation}{section}

\graphicspath{{./Imgs/}}

\newcommand{\R}{\mathbb R}
\newcommand{\C}{\mathbb C} 

\newcommand{\be}{\begin{equation}}
\newcommand{\ee}{\end{equation}}
\newcommand{\ba}{\begin{eqnarray}}
\newcommand{\ea}{\end{eqnarray}}
\newcommand{\beq}{\begin{equation}}
\newcommand{\eeq}{\end{equation}}

\usepackage{color}
\usepackage{stmaryrd}

\newcommand{\norm}[1]{\left\Vert#1\right\Vert}
\newcommand{\D}{\displaystyle}

\newcommand{\T}{{\LARGE$\tau$\,}}

\newcommand{\n}{n}
\newcommand{\Bi}[2]{\int\limits_{\partial\Omega}#1#2}

\renewcommand{\S}{\mathcal S}
\newcommand{\N}{\mathcal N}

\title[Control of underwater vehicles]{Control of underwater vehicles in inviscid fluids--II. Flows with vorticity}

\author{Rodrigo Lecaros}
\address{CMM - Centro de Modelamiento Matem\'atico. 
Universidad de Chile (UMI CNRS 2807), 
\hfill\break\indent 
Avenida Blanco Encalada 2120, Casilla 170-3, Correo 3, Santiago, Chile\medskip\medskip\medskip}
\email{rlecaros@dim.uchile.cl}
\author{Lionel Rosier}
\address{Centre Automatique et Syst\`emes\\ MINES ParisTech, PSL Research University\\
60 Boulevard Saint-Michel\\ 75272 Paris Cedex 06, France  
\medskip\medskip\medskip}
\email{Lionel.Rosier@mines-paristech.fr}

\begin{document}
\begin{abstract}
In a recent paper, the authors investigated the controllability of an underwater vehicle immersed in an infinite volume of an 
inviscid fluid, assuming that the flow was irrotational. The aim of the present paper is to pursue this study by considering the more general case
of a flow {\em with vorticity}.  It is shown here that the local controllability of the position and the velocity of the underwater vehicle (a vector in $\R ^{12}$) 
 holds in a flow {\em with vorticity}  whenever it holds in a flow {\em without vorticity}.  
\end{abstract}

\maketitle

\section{Introduction}
An accurate model for the motion of a boat (without rudder)  equipped with tunnel thrusters was  investigated in \cite{GR}. 
In that paper, using Coron's return method (see \cite{coron-book}), the authors proved that it was in general possible to control 
both the position and the velocity of the boat (a vector in $\R ^6$) by using {\em two} control inputs. 
The fluid 
was assumed to be inviscid, but not necessarily irrotational, and its motion was described by Euler equations for incompressible fluids. 

In \cite{LR}, the authors started the study of the controllability of an underwater vehicle $\mathcal S$  (e.g. a submarine) immersed in an infinite volume
of an inviscid fluid (filling $\R ^3\setminus {\mathcal S}$). Assuming that the fluid was irrotational, they proved by using  Coron's return method 
the controllability of both the position and the velocity of the vehicle (a vector in $\R^{12}$) 
by using $6$, or $4$, or merely $3$ control inputs for appropriate geometries. The aim of the present paper is to pursue this study by considering the more general 
case of a flow {\em with vorticity}. We will show that the local controllability of both the position and the velocity of the underwater vehicle holds in a flow {\em with vorticity}
whenever it holds in a flow {\em without vorticity}. The method of proof is inspired by the one of \cite{GR}:  the extension of the exact controllability to a system with a  (small) vorticity
is achieved  by a perturbative approach relying on a topological argument. Next, the small vorticity assumption is removed by using a scaling argument.  However, to prove the wellposedness of the complete system we shall use here the contraction mapping theorem 
instead of the Schauder fixed-point theorem as in \cite{GR}. This choice leads to a more straightforward proof.

Our fluid-structure interaction system can be described as follows. The underwater vehicle, represented by a rigid body occupying a connected 
compact set $\S (t)\subset \R ^3$, is surrounded by a homogeneous incompressible perfect fluid filling the open set $\Omega (t) :=\R ^3\setminus \S (t)$ 
(as e.g. for a submarine immersed in an ocean).
We assume that $\Omega (t)$ is $C^\infty$ smooth and connected.
 Let  $\S =\S (0)$ and 
 $\Omega(0)= \R^3\setminus \S (0)$ 
 denote  the initial configuration ($t=0$). Then, the dynamics of the fluid-structure system are governed by the following system of PDE's 

\begin{eqnarray}\label{euler1}
\D\frac{\partial u}{\partial t}+(u\cdot\nabla)u+\nabla p=0,&& t\in(0,T), \; x\in\Omega(t),\qquad \\ 
\textrm{div } u =0,&& t\in(0,T), \; x\in\Omega(t),\qquad \\ 
u\cdot \n=(h'+\zeta  \times  (x-h))\cdot\n+w(t,x),&& t\in(0,T),\; x\in\partial\Omega(t),\qquad \\ 
\lim\limits_{|x|\to+\infty}u(t,x)=0,&& t\in (0,T),  \\ 
m_0h''=\D\int\limits_{\partial\Omega(t)} p\n \,d\sigma,&&t\in(0,T), \\ 
\D\frac{d}{dt}(QJ_0Q^\ast \zeta)=\D\int\limits_{\partial\Omega(t)} (x-h)\times p\n \,d\sigma ,&&t\in(0,T),\\ 
\label{eq for Q}
Q'= S(\zeta)Q ,&&t\in(0,T),\\ 
u(0,x)=u_0(x),&& x\in \Omega  (0), \\
 \label{original_system_f}
 (h(0),Q(0),h'(0),\zeta (0))=(h_0,Q_0,h_1,\zeta _0)&\in&\R^3\times \text{SO}(3)\times \R^3\times \R ^3.\label{euler6}
 \end{eqnarray}
In the above equations, $u$ (resp. $p$) is the velocity field (resp. the pressure) of the fluid, $h$ denotes the position of the center of mass of the solid, $\zeta$ denotes the angular velocity and $Q \in \R ^{3\times 3} $ the rotation matrix giving the orientation of the solid. 
The positive constant $m_0$ and the matrix $J_0$, which stand for 
the mass and the inertia matrix of the rigid body, respectively,
are defined  as
$$m_0= \int\limits_{\S} \rho(x)dx,\;\;\;J_0=\D\int\limits_{\S}\rho(x)(|x|^2 Id- xx^\ast)dx, $$
where $\rho(\cdot)$ represents the density of the rigid body. 
The vector $\n$  is the outward unit vector to $\partial\Omega(t)$,
$x\times y$ is  the cross product between the vectors $x$ and $y$,  and $S(y)$ is the skew-adjoint matrix such that
$S(y)x=y\times x$, i.e.
$$
S(y)=
\left( \begin{array} {ccc}
0  & -y_3 & y_2 \\
y_3 & 0 & -y_1 \\
-y_2 & y_1 & 0
\end{array}
\right) .
$$
The neutral buoyancy condition reads
\be
\label{buoyancy}
 \int\limits_{\S} \rho(x)dx = \int\limits_{\S} 1dx.
\ee

When $f$ is a function depending on $t$, $f'$ (or $\dot f$) stands for  the derivative of $f$ with respect to  $t$. 
For $A\in \R ^{M\times N}$ ($M,N\in {\mathbb N}^\ast $),  
$A^\ast$ denotes the transpose of the matrix $A$, and $Id$ denotes the identity matrix. The term $w(t, x)$, which 
stands for the flow through the boundary of the rigid body, is taken as control input. Its support will be strictly included in $\partial\Omega(t)$, 
and actually only a finite dimensional control input will be considered here (see below (\ref{form_control}) for the precise form of 
the control term $w(t, x)$).

When no control is applied (i.e. $w(t,x)=0$), then the existence
and uniqueness of strong solutions to \eqref{euler1}-\eqref{euler6} was
obtained first in \cite{ORT1} for a ball embedded in $\R^2$, and next in \cite{ORT2} for a rigid body $\S$ of
arbitrary form (still in $\R ^2$).  The case of a ball in $\R ^3$ was investigated in  \cite{RR2008}, 
and the case of a rigid body of arbitrary form in $\R ^3$ was studied in \cite{WZ}.  (See also \cite{sueur} for the motion of a rigid body in the inviscid limit of Navier-Stokes equations 
and \cite{GST} for the time regularity of the flow.) The detection of a rigid body ${\mathcal S} (t)$ from a partial measurement of the fluid velocity (or of the pressure) has been tackled in 
\cite{CCOR} when $\Omega (t) = \Omega _0 \setminus {\mathcal S} (t)$ ($\Omega _0\subset \R ^2$  denoting a fixed cavity) and in \cite{CMM} 
when $\Omega (t) = \R ^2\setminus {\mathcal S} (t) $.   

Note also that since the fluid is flowing through a part of the boundary of the rigid body, additional  boundary conditions 
are needed to ensure the uniqueness of the solution of \eqref{euler1}-\eqref{euler6} (see \cite{Yudovich64}, 
\cite{Kazhikhov}). In dimension three, one can specify the tangent components of the 
vorticity $\omega (t, x) := \textrm{curl }v(t, x)$ on the inflow section; 
that is, one can set
\beq
\omega (t,x)\cdot \tau _i = g_0(t,x)\cdot \tau _i \;\;\textrm{for }w(t,x) < 0,\ i=1,2,
\eeq
where $g_0(t, x)$ is a given function and $\tau_i$, $i=1,2$,  are linearly independent vectors tangent to 
$\partial \Omega (t)$. On the other hand, since $ \omega $ is divergence-free in $\Omega$, we have that $\int_{\partial \Omega (t) } \omega (t,x) \cdot n\,  d\sigma =0$.

In order to write the equations of the fluid in a {\em fixed frame}, we perform a change of coordinates. We set 
\begin{eqnarray}
x&=&Q(t)y+h(t),\\ 
v(t,y)&=&Q^\ast(t) u(t,Q(t)y+h(t)),\\ 
{\bf q(}t,y)&=&p(t,Q(t)y+h(t)),\\
\label{def l}
l(t)&=&Q^\ast(t)h'(t), \\
\label{def r} 
r(t)&=&Q^\ast (t) \zeta (t).
\end{eqnarray}
Then  $x$ (resp. $y$) represents the vector of coordinates of a point in a fixed frame (respectively in 
a frame linked to the rigid body). 
Note that, at any given time $t$, $y$ ranges over the fixed domain 
\[ 
\Omega  := Q_0^* ( \Omega (0) - h_0 )
\]
when $x$ ranges over $\Omega(t)$. Finally, we assume that the control takes the form

\beq\label{form_control}
w(t, x) =  w(t, Q(t)y + h(t)) =\sum_{j=1}^m w_j (t)\chi_j(y),
\eeq
where $m \in {\mathbb N} ^\ast$ stands for the number of independent inputs, and $w_j(t) \in \R$ is the control 
input associated with the function $\chi_j \in C^\infty(\partial\Omega)$. To ensure the conservation of the mass of the fluid, we impose the relation

\beq
\Bi{\chi_j(y)d\sigma}=0\; \textrm{ for }1\leq j\leq m.
\eeq

Then the functions $(v, {\bf q}, l, r)$ satisfy the following system

 \begin{eqnarray} 
 \label{S1}
\D\frac{\partial v}{\partial t} +((v-l-r\times y)\cdot\nabla)v+r\times v+ \nabla {\bf q}=0,&& t\in(0,T),\; y\in\Omega , \qquad \qquad \\ 
\label{S2} 
\textrm{div }v=0,&& t\in(0,T),\; y\in\Omega ,\qquad \\ 
\label{S3}
\D v\cdot \n=(l+r\times y)\cdot\n+\sum\limits_{1\leq j\leq m}w_j(t)\chi_j(y),&& t\in(0,T),\; y\in\partial\Omega , \qquad \qquad \\ 
\label{S4}
\lim\limits_{|y|\to+\infty}v(t,y)=0,&& t\in (0,T), \qquad \qquad \\
\label{S5}
m_0\dot l=\D\int\limits_{\partial\Omega} {\bf q}\n \,d\sigma -m_0r\times l,&&t\in(0,T),\\ 
\label{S6}
J_0 \dot r=\D\int\limits_{\partial\Omega} {\bf q}(y\times \n) \,d\sigma-r\times J_0r,&&t\in(0,T),\\
 \label{S7}
(l(0),r(0))= (l_0,r_0) := (Q_0^\ast h_1, Q_0^\ast \zeta _0), &&v(0,y)=v_0(y):=Q_0^\ast   u_0(Q_0 y+h_0).\qquad \qquad 
\end{eqnarray}

The initial velocity field $v_0\in C^{2,\alpha}(\overline{\Omega})$ has to satisfy the following compatibility conditions
\beq
\label{qwe}
\left\{
\begin{array}{rlll}
\textnormal{curl}\,v_0&=&\omega_0 & \textnormal{in } \Omega,\\
\textnormal{div}\,v_0&=&0 & \textnormal{in } \Omega,\\
v_0\cdot\n&=&(l_0+r_0\times y)\cdot \n+\sum\limits_{1\leq j\leq m}w_j(0)\chi_j(y) & \textnormal{on } \partial\Omega,\\
\lim\limits_{|y|\to+\infty}v_0(y)&=&0&
\end{array}
\right.
\eeq
where $\omega _0 : =\textrm{curl } v_0$ is the {\em initial vorticity}.

Once  $(l,r)$ is known, the motion of the underwater vehicle is described by the system
\ba
Q'(t)&=&Q(t)S(r(t)),\\
h'(t)&=&Q(t)l(t), \\
\zeta(t)&=&Q (t)  r(t).
\ea

Using quaternions, the rotation matrix $Q$ can be parametrized by 
\[ \vec{q}\in  B_1(0) := \{\vec{q}=(q_1,q_2,q_3)\in\R^3;\;\; | \vec{q} | := \sqrt{q_1^2 + q_2^2 +q_3^3}  <1\} \]
 (see e.g. \cite{LR}); 
namely, we can write $Q=R(\vec{q})$ where 
\[
R(\vec{q}) := \left( 
\begin{array}{ccc}
q_0^2 +q_1^2 -q_2^2 -q_3^2   &  2(q_1q_2-q_0q_3)                       &  2(q_1q_3+q_0q_2)                    \\
2(q_2q_1+q_0q_3)                      &  q_0^2 -q_1^2 +q_2^2 -q_3^2   &  2(q_2q_3 -q_0q_1)                    \\
2(q_3q_1-q_0q_2)                       &  2(q_3q_2 + q_0q_1)                    &  q_0^2 -q_1^2 -q_2^2 +q_3^2
\end{array}
\right) .
\]
with  $\vec{q}=(q_1,q_2,q_3)\in B_1(0)$ and  $q_0 :=\sqrt{1- | \vec{q} | ^2}$. Let $Q_0=R(\vec{q}_0)$ with 
$\vec{q}_0 = (q_{1,0}, q_{2,0}, q_{3,0})$. 

Then  the dynamics of $\vec{q}$ and $h$  are given by

\beq\label{systempq}
\left\{\begin{array}{ccl}
h'(t)&=&  (1- |  \vec{q} | ^2) l+ 2\sqrt{1- | \vec{q} | ^2}\,   \vec{q}\times l  + (l\cdot  \vec{q}\, ) \vec{q} -  \vec{q} \times l\times  \vec{q}, \\[3mm]
{ \vec{q}\, }'(t)&=& \frac{1}{2} (\sqrt{1- | \vec{q} |^2} \, r +  \vec{q} \times r ), \\[3mm]
h(0)&=&h_0, \quad \vec{q}(0)=\vec{q}_0.\\[3mm] 
\end{array}\right. 
\eeq

When there is no vorticity ($\omega \equiv 0$), sufficient conditions of local exact controllability 
for $(h, \vec{q}, l, r)$ were derived in \cite[Theorem 3.10]{LR}.  That result was applied to the controllability of an ellipsoidal  submarine with a small
number of controls: $m\in \{ 3, 4 ,6 \}$. The reader is referred to \cite{LR} for precise statements. 
The method of proof of  \cite[Theorem 3.10]{LR}, inspired by the one of \cite[Theorem 2.1]{GR}, combined Coron's return method (see \cite{coron-book}), the flatness approach
(for the construction of the reference trajectory) and a variant of Silverman-Meadows criteria.   
In the following, we shall assume that the conclusion of   \cite[Theorem 3.10]{LR} (controllability
without vorticity) holds; namely, \\[3mm]
(H) {\em For any $T>0$, there exist a number $\eta >0$ and a map  $W \in C^1(B_{\R ^{24}}(0, \eta ), C^1( [0,T]; \R ^m))$ which associates with 
any $(h_0,\vec{q}_0,l_0,r_0,h_T, \vec{q}_T, l_T, r_T)\in B_{\R ^{24}}(0, \eta )$ a control $w \in C^1([0,T], \R ^m)$ with $w(0)=0$ steering
the state of system \eqref{S1}-\eqref{S7} and \eqref{systempq}
{\em without vorticity} from $(h_0,\vec{q}_0,l_0,r_0)$ at $t=0$ to $(h_T,\vec{q}_T,l_T,r_T)$ at $t=T$.}\\[1mm]  

In (H), we used the obvious notation: $B_{\R ^N} (0,\eta ) :=\{ x\in \R^N; \ |x| <\eta \}$. 

The aim of this work is to extend this property to the more general case of fluids with vorticity. Here, we shall use the contraction mapping theorem (instead of a compactness approach as in 
\cite{GR}) to obtain in a direct way the existence and uniqueness of the solution of \eqref{S1}-\eqref{S7}.
The main result in this paper
is the following

\begin{theorem} 
\label{thm1}
Assume that the assumption (H) is fulfilled, and pick any
$T_0>0$. Then there exists  $\eta >0$ such that for any $(h_0,\vec{q}_0,l_0,r_0)\in \R ^{12}$
and any $(h_T,\vec{q}_T,l_T,r_T)\in \R ^{12}$ with
\[
|(h_0,\vec{q}_0)| < \eta,\quad |(h_T,\vec{q}_T)|< \eta ,
\]
and for any $\omega_0\in C^{1,\alpha}(\overline{\Omega})\cap 
M^p_{1, \delta + 2} \cap M^p_{0, \delta + 3}$ (see below for the definition of these spaces) with
\begin{eqnarray*}
&&\vert \omega _0 (y^1)-\omega _0(y^2) \vert \le \frac{K}{ [1 + \min(\vert y^1 \vert, \vert y^2 \vert )]^\kappa } \vert y^1-y^2 \vert, \quad \forall (y^1,y^2)\in \Omega ^2 , \\
&&\vert \frac{\partial \omega _0}{\partial y} \vert = O(\vert  y\vert ^{-1} ) \quad \text{ as }  \vert y\vert \to +\infty , \\
&&  \vert \frac{\partial \omega _0}{\partial y}  (y^1)- \frac{\partial \omega _0}{\partial y} (y^2) \vert \le \frac{K}{1 + \min (\vert y^1 \vert, \vert y^2 \vert )} \vert y^1-y^2 \vert, \quad \forall 
(y^1,y^2)\in \Omega ^2 
\end{eqnarray*}
for some constants $p \in (3,4]$, $\delta \in [0, 1-\frac{3}{p} )$, $\alpha\in ( 0,1-\frac{3}{p} ]$, 
$\kappa >  3 + \delta  + \frac{3}{p}$ and  $K>0$, 
if $v_0$ denotes the solution of \eqref{qwe} with $w_j(0)=0$ for $1\le j\le m$, then
there exist a time $T\in (0,T_0]$ and a control input $w\in C^1([0,T] ;\R ^m)$
with $w(0)=0$ such that the system \eqref{S1}-\eqref{S7} and \eqref{systempq}
admits a solution $(h,\vec{q} ,l,r,v,{\mathbf q})$ satisfying 
\begin{equation*}
(h,\vec{q},l,r)_{|t=T}=(h_T,\vec{q}_T,l_T,r_T).
\end{equation*}
\end{theorem}

\begin{remark}
In our previous control result \cite[Theorem 3.10]{LR} for a system without vorticity, it was required that the 
initial/final velocities be small, but this restriction could easily be removed by using the same scaling argument as in the proof of Theorem 
\ref{thm1}.  
\end{remark} 
The paper is organized as follows. In Section 2 we prove the existence and uniqueness of the solution of the control problem \eqref{S1}-\eqref{S7} (the vorticity being extended to $\R^3$) by applying
the contraction mapping theorem in Kikuchi's spaces. The proof of Theorem \ref{thm1} is given in Section 3.  
\section{Wellposedness of the system with vorticity}
Let us introduce some notations.
For $k\in {\mathbb N}$ and $\alpha\in(0,1)$, let $C^{k,\alpha}(\overline{\Omega})$ denote the classical H\"older space endowed with the norm
$$ \|f\|_{C^{k,\alpha}(\overline{\Omega})}=\sum_{
\tiny \begin{array}{c} 
\beta = (\beta _1, \beta _2 , \beta _3) \in{\mathbb N}^3\\
\beta _1 + \beta _2 + \beta _3 \leq k
\end{array}}\Big(\|\partial^\beta f\|_{L^\infty(\Omega)}+|\partial^\beta f|_{0,\alpha}\Big),$$
where
$$ |f|_{0,\alpha}=\sup\left\{\frac{|f(x)-f(y)|}{|x-y|^\alpha}\;\;;\;x\in\overline{\Omega},\; y\in\overline{\Omega},\; x\neq y\right\}.$$
We also need  some notations borrowed from \cite{kikuchi86}. 
Let $\langle y\rangle  = (1 + |y|^2 )^\frac{1}{2}$.  
For $s\in \mathbb N$, $p\in[1,\infty )$ and $\lambda \ge 0$,  let $M^p_{s, \lambda} $ denote the completion of the space of functions
in $C^\infty (\overline{\Omega})$ with compact support in $\overline{\Omega}$ for the norm 
\[
\Vert u\Vert _{M^p_{s, \lambda }} = 
\sum_{
\tiny \begin{array}{c} 
\beta = (\beta _1, \beta _2 , \beta _3) \in{\mathbb N}^3\\
\beta _1 + \beta _2 + \beta _3 \leq s
\end{array}}
 \Vert \langle y\rangle  ^{\lambda + \beta _1+\beta  _2+\beta  _3 }  \partial ^\beta u \Vert _{ L^p ( \Omega )}.  
\] 
In particular, for $s=0$, $\Vert u\Vert _{M^p_{0, \lambda }} = \Vert u\Vert _{L^p_{p\lambda}} := (\int_\Omega |u|^p \langle y\rangle  ^{p\lambda }dy) ^\frac{1}{p}$. 
We shall mainly use the space $M^p_{1,\lambda}$ (for the vorticity) and  $M^p_{2,\lambda }$ (for the velocity) endowed with the respective norms
\begin{eqnarray}
\Vert u\Vert _{M^p_{1, \lambda }} &=&  \Vert \langle y\rangle  ^{\lambda }  u \Vert _{ L^p ( \Omega )} + 
\sum_{1\le i\le 3}  \Vert \langle y\rangle  ^{\lambda + 1}  \partial _{y_i} u \Vert _{ L^p ( \Omega )} , \\  
\Vert u\Vert _{M^p_{2, \lambda }} &=& 
\Vert \langle y\rangle  ^{\lambda }  u \Vert _{ L^p ( \Omega )} + 
\sum_{1\le i\le 3}  \Vert \langle y\rangle  ^{\lambda + 1}  \partial _{y_i} u \Vert _{ L^p ( \Omega )}  + 
\sum_{1\le i,j\le 3} \Vert \langle y\rangle  ^{\lambda + 2}  \partial _{y_j} \partial _{y_i}  u \Vert _{ L^p ( \Omega )}.  
\end{eqnarray}

Let $\pi$ be a continuous linear extension operator from functions defined in $\Omega$ to functions defined in $\R^3$, which maps
 $C^{k,\alpha}(\overline{\Omega})$ to $C^{k,\alpha}(\R^3)$ for all $k \in \mathbb N$ and all $\alpha \in (0,1)$. The construction of such 
 a ``universal" extension operator is classical, see e.g.  \cite[p. 194]{stein}. We may also ask that $\pi$ preserves the divergence-free character, see  e.g. \cite{KMPT}.
 
We introduce some functions $\phi _i$, $i=1,2,3$, $\varphi _i$, $i=1,2,3$, and $\psi _j$, $j=1,2, ... , m$, satisfying 
\beq
\label{eq for phi}
\Delta\phi_i =\Delta\varphi_i =\Delta\psi_j =0 \qquad \textrm{ in } \Omega, 
\eeq
\beq\label{eq for phi:01}
\D\frac{\partial \phi_i}{\partial \n} = \n _i,\;\;\D\frac{\partial \varphi_i}{\partial \n} 
= (y\times \n)_i,\;\; \D\frac{\partial \psi_j}{\partial \n} = \chi_j  \qquad  \textrm{ on }\partial \Omega ,
\eeq
\beq\label{eq for phi:02}
\D\lim\limits_{|y|\to+\infty}\nabla \phi _i(y)=0,\;\; \D\lim\limits_{|y|\to+\infty}\nabla \varphi _i(y)=0,\;\;\D\lim\limits_{|y|\to+\infty}\nabla \psi _j(y)=0 .
\eeq

As the open set $\Omega$ and the functions $\chi_j$, $1 \leq j \leq m$, 
supporting the control are assumed to be smooth, we infer that the functions $\nabla\phi_i$ ($i = 1,2,3$), 
the functions $\nabla \varphi _i$ ($i=1,2,3$) and the functions $\nabla\psi_j$ ($1 \leq j \leq m$) belong to $H^{\infty}(\Omega)$. On the other hand, 
it follows from \cite[Proof of Lemma 2.7]{kikuchi86} that for all $\alpha =(\alpha _1, \alpha _2, \alpha _3) \in 
{ \mathbb N }^3$ with $\alpha _1 + \alpha _2 + \alpha _3 \ge 1$, we have
\begin{equation}
\label{estimation1}
\vert \partial ^\alpha \phi _i (y) \vert + \vert \partial ^\alpha \varphi _i (y) 
 \vert + \vert \partial ^ \alpha \psi _j (y)  \vert \le C  \langle y \rangle ^{- 1 - ( \alpha _1 + \alpha _2+ \alpha _3) }, \qquad 
i=1,2,3, \ j=1,2,...,m , \ y\in \Omega .
\end{equation}

For notational convenience, in what follows $\int_{\Omega} f$ (resp. $\int_{\partial \Omega} f$) stands for $\int_{\Omega } f(y)dy$  
(resp. $\int_{\partial \Omega} f(y) d\sigma (y)$).
  
Let us introduce the matrices  $M,J,N\in \R^{3\times 3}$,  defined by
\beq \label{matrix 1}
M_{i,j} =\int\limits_{\Omega}\nabla\phi_{i} \cdot \nabla\phi_{j}  =\Bi{\n_{i}\phi_{j}}=\Bi{\frac{\partial \phi_{i}}{\partial \n}\phi_{j}},
\eeq
\beq J_{i,j}  =\int\limits_{\Omega}\nabla\varphi_{i} \cdot \nabla\varphi_{j}  =\Bi{(y\times \n)_{i}\varphi_{j}}=\Bi{\frac{\partial \varphi_{i}}{\partial \n}\varphi_{j}},
\eeq

\beq N_{i,j}=\int\limits_{\Omega}\nabla\phi_{i}\cdot \nabla\varphi_{j}=\Bi{\n_{i}\varphi_{j}}=\Bi{\phi_{i}(y\times \n)_{j}}.
\eeq
Next we define the
matrix $\mathcal{J} \in \R^{6\times6}$ by
\beq\label{matrix J}
\mathcal{J}=
\left(\begin{array}{cc}
m_0\,  Id & 0\\
0 & J_{0}
\end{array}\right)
+\left(\begin{array}{cc}
M & N\\
N^{\ast} & J
\end{array}\right).
\eeq
It is easy to see that $\mathcal{J}$  is a (symmetric)  positive definite matrix.

For a potential flow (i.e. without vorticity),    the dynamics of $(l,r)$ are given by
\beq\label{system lr}
\left(\begin{array}{c}
l\\
r
\end{array}\right)'=\mathcal{J} ^{-1}(C w'+F(l,r,w)),
\eeq
where 
\beq\label{non lineal function}
\begin{array}{rcl}
F(l,r,w)&=& -\left(\begin{array}{cc} S(r)  & 0\\ \\ S( l )  & S(r) \end{array}\right)\left(\mathcal{J}\left(\begin{array}{c} l\\r\end{array}\right) -Cw\right) 
-\sum\limits_{p=1}^m w_{p}
\left(\begin{array}{c} \D  L^M_{p}l+R^M_{p}r+W^M_{p}w \\ \\ \D  L^J_{p}l+R^J_{p}r+W^J_{p}w 
\end{array}\right),
\end{array}
\eeq
and 
\beq\label{matrix C}
C=-\left(\begin{array}{c}C^{M}\\C^{J}\end{array}\right),
\eeq
with 
\beq (C^{M})_{i,j}=\int\limits_{\Omega}\nabla\phi_{i} \cdot \nabla\psi_{j}=\Bi{\n_{i}\psi_{j}}=\Bi{\phi_{i}\chi_{j}},
\eeq

\beq (C^{J})_{i,j}=\int\limits_{\Omega}\nabla\varphi_{i} \cdot \nabla\psi_{j}=\Bi{(y\times \n)_{i}\psi_{j}}=\Bi{\varphi_{i}\chi_{j}},
\eeq

\ba
 (L^M_{p})_{i,j}=\D\Bi{(\nabla\phi_{j})_{i}\chi_{p}}, &&(L^J_{p})_{i,j}=\D\Bi{(y\times \nabla\phi_{j})_{i}\chi_{p}},  \\
 (R^M_{p})_{i,j}=\D\Bi{(\nabla\varphi_{j})_{i}\chi_{p}}, && (R^J_{p})_{i,j}=\D\Bi{(y\times \nabla\varphi_{j})_{i}\chi_{p}}, \\ 
 (W^M_{p})_{i,j}=\D\Bi{(\nabla\psi_{j})_{i}\chi_{p}}, && (W^J_{p})_{i,j}=\D\Bi{(y\times \nabla\psi_{j})_{i} \chi_{p}}. \label{matrix 6}
\ea
We refer the reader to \cite{LR} for the derivation of \eqref{system lr}.

The first main result in this paper is a local existence result.
\begin{theorem}\label{theoreme1}
Let  $p \in (3,4]$, $\delta \in [0, 1-\frac{3}{p} )$, $\alpha\in ( 0,1-\frac{3}{p} ]$,  and $T>0$. Assume given $(l_0,r_0)\in\R^3\times\R^3$ and  $\omega_0\in C^{1,\alpha}(\overline{\Omega})\cap 
M^p_{1, \delta + 2} \cap M^p_{0, \delta + 3}$ with
\begin{eqnarray}
&&\vert \omega _0 (y^1)-\omega _0(y^2) \vert \le \frac{K}{ [1 + \min(\vert y^1 \vert, \vert y^2 \vert )]^\kappa } \vert y^1-y^2 \vert, \quad \forall (y^1,y^2)\in \Omega ^2 ,
\label{Homega} \\
&&\vert \frac{\partial \omega _0}{\partial y} \vert = O(\vert  y\vert ^{-1} ) \quad \text{ as }  \vert y\vert \to +\infty \label{zzz1}, \\
&&  \vert \frac{\partial \omega _0}{\partial y}  (y^1)- \frac{\partial \omega _0}{\partial y} (y^2) \vert \le \frac{K}{1 + \min (\vert y^1 \vert, \vert y^2 \vert )} \vert y^1-y^2 \vert, \quad \forall 
(y^1,y^2)\in \Omega ^2 
\label{zzz2}
\end{eqnarray}
for some constants $\kappa >  3 + \delta  + \frac{3}{p}$ and  $K >0$. 
Let also a control input $w\in C^1([0,T],\R^m)$ be given. 
Assume that the initial velocity field $v_0\in C^{2,\alpha}(\overline{\Omega})$ fulfills the following compatibility conditions
\beq
\left\{
\begin{array}{rlll}
\textnormal{curl}\,v_0&=&\omega_0 & \textnormal{in } \Omega,\\
\textnormal{div}\,v_0&=&0 & \textnormal{in } \Omega,\\
v_0\cdot\n&=&(l_0+r_0\times y)\cdot \n+\sum\limits_{1\leq j\leq m}w_j(0)\chi_j(y) & \textnormal{on } \partial\Omega,\\
\lim\limits_{|y|\to+\infty}v_0(y)&=&0.&
\end{array}
\right.
\eeq
Then we can find a time $T'\in(0,T]$ satisfying $CT'<1$, where \[C=C(\Vert \omega _0\Vert _ {C^{1, \alpha } (\overline{\Omega})} + \Vert \omega _0\Vert _{M^p_{1, \delta +2}} , \vert l_0\vert , \vert r_0\vert ,
\Vert w\Vert _{C^1([0,T])})\] is nondecreasing in all its arguments, 
and a solution $(v, {\bf q}, l, r)$ of \eqref{S1}-\eqref{S7} in the class
\beq\label{clas01}
v\in C([0,T'];C^{2,\alpha}(\overline{\Omega}) \cap M^p_{2,\delta + 1} ),\;\; \nabla v\in C([0,T'];L^4_{p(\delta + 2)}(\Omega)),
\eeq
\beq\label{clas02}
\nabla {\bf q}\in C([0,T'];L^2(\Omega)),
\eeq
\beq
\lim_{ |y| \to +\infty} \nabla {\bf q}(t,y)=0, \quad \forall t\in [0,T'],  
\eeq
\beq\label{clas03}
(l,r)\in C^1([0,T'];\R^3\times\R^3).
\eeq
Moreover, for $\Vert w\Vert _{ C^1 ( [0,T] ) } \le R$ ($R>0$ being any given constant), this solution satisfies 
\beq
\norm{(l,r)-(\overline{l},\overline{r})}_{L^\infty(0,T')}+\norm{v-\overline{v}}_{L^\infty(0,T';C^{2,\alpha}(\overline{\Omega})) }
\leq C' \Big(\norm{\omega_0}_{C^{1,\alpha}(\overline{\Omega})}+\norm{\omega_0}_{M^p_{1,\delta + 2}}\Big),
\label{aaa1}
\eeq
for some constant $C' > 0$, 
where $(\overline{l},\overline{r},\overline{v})$ is the potential solution of \eqref{S1}-\eqref{S7} associated with $l_0, r_0$,
$\{w_j\}_{1\leq j\leq m}$, and
$\overline{\omega}_0 = 0$.
\end{theorem}

\begin{remark}
\begin{enumerate}
\item
Note that, using the mean-value theorem, the assumption \eqref{zzz1} implies \eqref{Homega} for $\kappa=1$, while $\omega _0\in M^p_{1, \delta +2}$ yields $\vert \omega _0(y)\vert \le O( |y| ^{-\delta -2} ) $ as $\vert y\vert \to +\infty$ by \cite[Lemma 2.2]{kikuchi86}.
\item Note that the fluid-structure system considered here is more complicated than those 
considered in \cite{RR2008, WZ}, for we have added a control input in the boundary condition \eqref{S3}. Moreover, we 
require the solution to be continuous with respect to the control input in order to  apply a perturbative argument at the end of the proof of Theorem \ref{thm1}. 
Inspired by the method developed in \cite{GR}, it is quite natural
to work in Kikuchi's spaces. Here, we shall prove the existence, uniqueness
and continuous dependence of the solution  with respect to the control input in {\em one step}, by using the contraction mapping 
theorem.  
\end{enumerate}
\end{remark}
Theorem \ref{theoreme1} will be established by using the contraction mapping principle (i.e. the Banach fixed-point theorem). We first define an operator whose fixed-points will give local-in-time solutions.

\subsection{The operator}
Let $p \in (3,4]$, $\delta \in [0, 1-\frac{3}{p} )$, $\alpha\in ( 0,1-\frac{3}{p}]$,  and $T>0$. We fix a control input $w\in C^1([0,T]; \R ^m)$. 
For $N>0$ and $P>0$ given, we introduce the set 
\ba
\label{a1}
{\mathcal F } &:=&\D\Big\{ (l,r,\omega) \in \R ^3 \times \R ^3 \times \big( C^{1,\alpha}(\overline{\Omega})\cap M^p_{1,\delta+ 2 } \big) ;\ 
\D  \vert l-l_0 \vert  +\vert  r-r_0\vert  \leq N, \nonumber\\
&&\qquad \|\omega\|_{C^{1,\alpha}(\overline{\Omega})} + \| \omega \|_{M^p_{1, \delta + 2}  }\leq P, \ \D\textnormal{div}\,\omega=0,\;\;\;  \int_{\partial\Omega}\omega\cdot\n \, d\sigma =0
\Big\} .
\ea
Then, using Arzela-Ascoli theorem for the restrictions to closed ball centered at the origin of partial derivatives of order one, 
 it is easy to see that ${\mathcal F}$ is a {\em closed} subset of the Banach space 
$E: = \R ^3 \times \R ^3 \times  \big( C^{0,\alpha}  (\overline{\Omega} ) \cap M^p_{0,\delta + 2} \big) $ endowed with the norm
\[
\Vert (l,r , \omega ) \Vert _E : = \vert l\vert + \vert r\vert + \Vert \omega \Vert_ {C^{0,\alpha} (\overline{ \Omega } ) }  + \Vert \omega \Vert _{L^p_{p(\delta + 2) } (\Omega )}\cdot 
\]   
It follows at once that for any $T'>0$, the set
 \[ 
  {\mathcal C} := \{  (l,r,\omega) \in C([0,T'],{\mathcal F}); \ (l(0),r(0),\omega (0) )=(l_0,r_0,\omega _0) \}  
  \]
   is a {\em closed} subset of the Banach space $C([0,T']; E)$ endowed with the norm
$\sup_{t\in [0,T']}  \Vert ( l(t) , r(t) , \omega ( t ) ) \Vert _E $,  which is also complete for the equivalent norm 
\be
\label{norme}
\interleave (l,r,\omega )  \interleave := \Vert l\Vert _{L^\infty (0,T')} + \Vert r\Vert _{L^\infty (0,T')} + \Vert \omega \Vert _{ L^\infty ( 0,T'; C^{0, \alpha }(\overline{\Omega} )   )} 
+ \Vert \omega \Vert _{ L^\infty (0,T'; L^p_{ p( \delta + 2) } (\Omega ) ) } .
\ee
Therefore, $ {\mathcal C}$  is {\em complete} for the distance associated with the norm $\interleave \cdot \interleave$.

Here, we pick
\beq 
\label{a2}
P =e ^e \cdot  ( C_6  \| \pi( \omega_0 ) \|_{C^{1,\alpha}(\R^3)} + C_7 \| \pi ( \omega _0 )\| _{  M^p_{1, \delta +2 } (\R ^3 )  } ) ,
\eeq
where $C_6$ and $C_7$ are some universal constants arising in the computations below and that we do not intend to give explicitly, and
$\Vert \cdot \Vert _{C^{1,\alpha} (\R ^3)}$ and $\Vert \cdot \Vert _{ M^p_{1, \delta +2} (\R ^3) }$ are defined as   
$\Vert \cdot \Vert _{C^{1,\alpha} (\overline{\Omega} )}$ and $\Vert \cdot \Vert _{ M^p_{1, \delta +2} }$, respectively. 

Let us now define the operator \T on $\mathcal{C}$: to any $(l, r, \omega) \in \mathcal{C}$, we associate
\beq
 \textnormal{\T} ( l , r , \omega ) := ( \hat{l} , \hat{r} , \hat{\omega} ) ,
\eeq
as follows. First, we introduce the ``fluid velocity" 
\beq \label{eq_velo}
v=\eta+\sum\limits_{i=1}^3l_i\nabla \phi_i+\sum\limits_{i=1}^3r_i\nabla \varphi_i+\sum\limits_{1\leq j\leq m}w_j(t)\nabla\psi_j,
\eeq
where $\eta$ is  the solution to the div-curl system (see e.g. \cite[Proposition 3.1]{kikuchi86})
\ba
\label{eqEta_1}
\textnormal{curl}\,\eta&=&\omega,  \quad (t,y)\in (0,T)\times\Omega, \\
\textnormal{div}\,\eta &=&0, \quad    (t,y) \in (0,T)\times\Omega, \\
\D \eta \cdot \n &=& 0, \quad    (t,y) \in (0,T) \times\partial\Omega,\\
\label{eqEta_4}
\lim\limits_{|y|\to+\infty}\eta(t,y) &=& 0, \quad  t\in (0,T), 
\ea

Next, we extend the velocity field and the initial vorticity by setting
\beq
\hat{v}(t,\cdot):=\pi[v(t,\cdot)],
\eeq
\beq
\hat{\omega}_0:=\pi[\omega_0].
\eeq

The flow $\hat{X}$ associated with $\tilde{v} : =\hat{v} - l - r\times y$ is defined as the solution to the Cauchy problem
\beq\label{flux of v tilde}
\left\{
\begin{array}{ccl}
\D\frac{\partial}{\partial s} \hat{X}(s;t,y)&=&\tilde{v} ( s, \hat{X} (s;t,y)) =  \hat{v}(s,\hat{X}(s;t,y))-l(s)-r(s)\times\hat{X}(s;t,y),\\\\
\hat{X}(t;t,y)&=& y.
\end{array}\right.
\eeq
The fact that $\hat{X}$ is defined globally on $[0,T']^2 \times \R ^3$ follows from the boundedness of $\hat{v}$ (see  below \eqref{cota_v}).
 
 We denote by $G$ the Jacobi matrix of $\hat{X}$, namely
 \beq\label{defG}
 G(s;t,y)=\frac{\partial\hat{X}}{\partial y}(s;t,y).
 \eeq
Differentiating in  \eqref{flux of v tilde} with respect to $y_{j}$ ($j=1,2,3$), we see that $G(s;t,y)$ satisfies the following system:
\ba
\frac{\partial G}{\partial s} 
&=& \D\frac{\partial\hat {v}}{\partial y}(s,\hat{X}(s ; t, y))\, G(s ; t,y) -r(s) \times G(s ; t, y),  \nonumber  \\
&=& \left( \D\frac{\partial\hat {v}}{\partial y}(s,\hat{X}(s ; t,y)) -S(r(s)) \right) G(s ; t, y),  \label{equation for G} \\
G(t; t,y)&=& Id\; \textrm{(the identity matrix)}.\nonumber 
\ea
We infer from  
\beq\label{div v tilde}
\textrm{div}(\tilde{v})=0 
\eeq
 that 
\beq\label{det G}
\textrm{det } G(s ;  t, y)=1.
\eeq
We now define
\beq\label{equation rot01}
\hat{\omega}(t,y) : =G^{-1}(0 ; t, y)\hat{\omega}_0(\hat{X}(0 ; t,y)).
\eeq
Finally, in order to define the pair $(\hat{l}, \hat{r})$, we introduce the function $\mu : [0,T']\times \Omega  \rightarrow\R$ (defined up to a function of $t$) which
solves 
\be\left\{\begin{array}{rrll}
-\Delta \mu&=&\D\textnormal{tr\,}\Big(\nabla v\cdot \nabla v\Big),&\textnormal{ in }(0,T)\times\Omega, \\
\D\frac{\partial \mu}{\partial n}&=&\D-\!\!\!\sum\limits_{1\leq j\leq m}\dot{w}_j(t)\chi_j(y)
-\left(\Big((v-l-r\times y)\cdot\nabla\Big)v+r\times v\right)\cdot \n,&\textnormal{on }(0,T)\times\partial\Omega ,\\ \\
\lim\limits_{\vert y\vert \to + \infty}\nabla\mu(t,y) &=&0 & \textnormal{in }(0,T).
\end{array}\right.
\label{system_mu}
\eeq
Note that  $\nabla \mu \in L^2(\Omega)$ if $\nabla v\in L^4(\Omega )\cap C^1(\overline{\Omega})$, and that, by Schauder estimates,  
$\mu \in C^{2,\alpha}_{loc}(\overline{\Omega })$  if in addition $v\in C^{2, \alpha}(\overline{\Omega})$. 
Then we define $\hat{l}$ and $\hat{r}$ by
\beq\label{sistema_l_r}
\left(\begin{array}{c}
\hat{l}(t) \\
\hat{r}(t)
\end{array}\right)
: =\left(\begin{array}{c}
l_0 \\
r_0
\end{array}\right)
+
\mathcal{J}^{-1}\int_0^t\left\{
\left(\begin{array}{c}
\Big(\int\limits_{\Omega} \nabla \mu(\tau,y)\cdot \nabla \phi_i(y) \,d y\Big)_{i=1,2,3}
\\
\Big(\int\limits_{\Omega} \nabla \mu(\tau,y)\cdot \nabla \varphi_i(y) \,d y\Big)_{i=1,2,3}
\end{array}\right) -
\left(\begin{array}{c}
m_0r\times l
\\
r\times J_0r
\end{array}\right)
\right\}d\tau .
\eeq
This completes the definition of {\LARGE$\tau$}.

\subsection{Fixed-point argument and local-in-time existence}

Our first step consists in proving the following result.

\begin{theorem}\label{theoreme2.2}  Let $N>0$ and $P>0$ be given. Then there exists some time $T' > 0$ such that \T is a contraction in  
$\mathcal{C}$ for the distance associated with $\interleave \cdot \interleave$. Thus, \T has a unique
fixed-point in $\mathcal C$. 
\end{theorem}

\begin{proof}[Proof of Theorem \ref{theoreme2.2}:] Set
$$\overline{N}:=N+|l_0|+|r_0|.$$
In the sequel, the various positive constants $C_i$ will depend on the geometry, on $\mathcal{J}$ and on the size of the controls $\|w_i\|_{C^1}$ only (hence, possibly also on $\pi$, but not on 
$T$, $l_0$, $r_0$, $\omega_0$, $N$, etc.).\\

{\bf Step 1.}  Let    $(l,r,\omega)\in\mathcal{C}$, and let $v$ be defined by  \eqref{eq_velo}.
It follows from \cite[Proposition 2.11]{kikuchi86} that for any $t\in [0,T']$, system \eqref{eqEta_1}-\eqref{eqEta_4} has a unique solution $\eta (t) \in M^p_{2,\delta + 1}$, and that 
\[
\Vert \eta (t) \Vert _{M^p_{2, \delta + 1}} \le C \Vert \omega (t) \Vert _{M^p_{1,\delta + 2}}.
\]
On the other hand, using \eqref{estimation1} and the fact that $0\le \delta <1-\frac{3}{p}$, it is easy to see that $\nabla\phi _i, \nabla \varphi _i, \nabla \psi _j\in M^p_{2, \delta +1}$
for $i=1,2,3$ and $j=1,2, ..., m$.  
It follows that 
\beq\label{eqCota01}
\norm{v(t)}_{M^p _{2, \delta + 1} } \leq C(  \Vert \omega (t) \Vert _{M^p_{1, \delta + 2} }  + \vert l(t) \vert + \vert r(t)\vert + \vert w(t) \vert  )  .
\eeq
Thus $v\in C( [0,T'] ; M^p_{2, \delta +1} )$ with 
\be
\label{X0} 
\Vert v\Vert _{L^\infty (0,T'; M^p_{2, \delta + 1})} \le C ( \overline{N} + P+ 1). 
\ee
Set 
\[
\N  : = \overline{N} + P +1 = N  +  | l_0 | + | r_0 | + P + 1. 
\]
Now Schauder estimates combined with the embedding $M^p_{2, \delta + 1} \subset C^1_b (\overline{\Omega})$ (see  \cite[Lemma 2.2]{kikuchi86})
 give that 
\begin{eqnarray*}
\norm{v (t) }_{C^{2,\alpha}(\overline{\Omega})} &\leq & C \left( \norm{\omega (t) }_{C^{1,\alpha}(\overline{\Omega})}+\norm{v (t) }_{C^{0,\alpha} (\overline{\Omega} ) }
 + \vert l (t) \vert +\vert r (t) \vert +\vert w (t) \vert  \right) \\
 &\leq  &  C \left( \norm{\omega (t) }_{C^{1,\alpha}(\overline{\Omega})}+\norm{v (t) }_{ M^p_{2, \delta + 1} }
 + \vert l (t) \vert  +\vert r (t) \vert  +\vert w (t) \vert  \right) \\
  &\leq &  C \left( \norm{\omega (t) }_{C^{1,\alpha}(\overline{\Omega})}+\norm{\omega (t) }_{M^p_{1, \delta + 2}}
 + \vert l (t) \vert  +\vert r (t) \vert  +\vert w (t) \vert  \right) .\\
\end{eqnarray*}
Thus $v\in C([0,T']; C^{2,\alpha }( \overline{\Omega} ) )$ with 
\ba
\norm{v}_{L^\infty(0,T';C^{2,\alpha}(\overline{\Omega}))} &\leq &  C \left(\norm{\omega}_{L^\infty(0,T';C^{1,\alpha}(\overline{\Omega}))}
+\norm{\omega }_{L^\infty(0,T'; M^p_{1, \delta + 2}) }  +\right.\nonumber\\
&&\quad  \left.
 + \norm{l}_{L^\infty(0,T')}+\norm{r}_{L^\infty(0,T')}+\norm{w}_{L^\infty(0,T')}\right).
\ea
		
Therefore, using the continuity of $\pi$,  we obtain that  
\beq\label{cota_v}
\norm{\hat{v}}_{L^\infty(0,T';C^{2,\alpha}(\R^3) ) }\leq\norm{\pi}C \N\leq C_1\N ,
\eeq
where $\norm{\pi}$ denotes the norm of $\pi$ as an operator in $\mathcal{L}(C^{2,\alpha}(\overline{\Omega}),C^{2,\alpha}(\R^3)).$\\

{\bf Step 2.} Let us turn our attention to $\hat X$ and $\hat \omega$. Taking the scalar product of each term of the first equation in  \eqref{flux of v tilde} by $\hat X$ results in 
\[
\vert \hat X\vert \frac{\partial \vert \hat X \vert  }{\partial s} =\frac{\partial}{\partial s} (\frac{1}{2} \vert \hat X\vert ^2) 
= \hat X \cdot  \frac{\partial \hat X}{\partial s } = (\hat v (s, \hat X) -l(s))\cdot \hat X \le \left( \vert \hat v (s, \hat X(s))  \vert + \vert l(s) \vert \right) \vert \hat X  \vert .  
\]
Simplifying by $\vert \hat X \vert $ and  using the second equation in  \eqref{flux of v tilde}, we obtain 
\be
\label{X1}
\left\vert \vert \hat X(s; t,y) \vert - \vert y\vert \right\vert \le  CT' \N . 
\ee
Thus $\hat X(s,t; \cdot) \not\in L^\infty(\R ^3)$ for all $(s,t)\in [0,T'] ^2$. 

It follows from \cite[Lemma 2.2]{kikuchi86} that any function $u\in M^p_{1, \delta +2}$ satisfies $\vert u(x) \vert = O(\vert x\vert ^{-\delta -1} )$ as $\vert x\vert \to  + \infty$, and
from \cite[Lemma 2.3]{kikuchi86} that 
$M^p_{1, \delta +2}$ is an algebra. Let $M^p_{1, \delta +2} (\R ^3) $ be defined as $M^p_{1, \delta +2}$ but with functions from $\R^3$ to $\R$.
We introduce the space
\[ V : =M^p_{1, \delta +2}  (\R ^3 ) ^{3\times 3} \oplus \R ^{3\times 3}\]
 with norm 
 \[
 \Vert G\Vert _V : = \Vert G_1\Vert _{M^p_{1, \delta +2} (\R ^3) } + \Vert G_2\Vert _{\R ^{3\times 3}}
 \]
 if $G=G_1+G_2$ with  $G_1 \in  M^p_{1, \delta + 2} (\R ^3 ) ^{3\times 3}$ and $G_2\in \R ^{3\times 3}$. (Note that $G_2$ is uniquely determined by $G$, as it is nothing but the $3\times 3$ matrix of the limits at  infinity of the entries of $G$.) Then it is easy to see that $V$ is a Banach space and an algebra.

Let us check that $(\partial \hat v/\partial y) (s, \hat X (s;t,y) ) \in L^\infty ((0,T')^2, V)$.   From  \eqref{X1} we have that
\[
\left\vert  \hat X(s;t,y)  \right\vert  \le \left\vert  y\right\vert   + CT' \N
\]
and proceeding as in \cite[p. 587-588]{kikuchi86}, we infer that for any $u\in M^p_{1, \delta + 2} (\R ^3) $
\be
\label{Y1}
\Vert  u(\hat  X(s; t, .)) \Vert _{M^p_{0, \delta + 2}  (\R ^3) } \le C (1+ [ C T' \N ] ^{ p(\delta + 2) } ) \Vert u \Vert _{M^p_{0, \delta + 2} (\R ^3) } .
\ee
On the other hand, using Gronwall's lemma in \eqref{equation for G}, we obtain with  \eqref{cota_v} that 
\be
\label{X2} 
\Vert G(s; t, .) \Vert _{L^\infty (\Omega )  }  \le C\exp (CT' \N). 
\ee
Since 
\[
\frac{\partial [u(\hat X(s;t,y)]}{\partial y_j} = \sum_{k=1}^3 \frac{\partial u}{\partial x_k} ( \hat X (s;t,y) ) \frac{\partial \hat X_k}{\partial y_j} (s;t,y),
\]
using \eqref{Y1} and \eqref{X2} for $\partial u/\partial y$ (with $\delta +3$ substituted to $\delta +2$), we arrive to 
\be
\label{Y2}
\Vert  \frac{\partial }{\partial y} [u(\hat  X(s; t,y)) ] \Vert _{M^p_{0, \delta + 3} (\R ^3) } 
\le C\exp (CT' \N)  \Vert \frac{\partial u}{\partial y} \Vert _{M^p_{0, \delta + 3} (\R ^3) } .
\ee
We infer from \eqref{X0} and \eqref{cota_v} that 
\[
\Vert \frac{\partial \hat v}{\partial y} \Vert _{L^\infty (0,T', M^p_{1, \delta +2} (\R ^3  )) } \le C \N.
\]
It follows that  $\frac{\partial \hat v}{\partial y} (s, \hat X(s; t,y)) -S(r(s)) \in L^\infty ( (0,T')^2, V)$ with 
 \be
 \Vert \frac{\partial \hat v}{\partial y} (s, \hat X(s; t,y)) -S(r(s)) \Vert _{V} \le C \N \exp (CT' \N ). 
 \ee
   
Solving  the linear Cauchy problem \eqref{equation for G} in the Banach algebra $V$, we see that $G\in C([0,T']^2 ; V)$ and (with Gronwall's lemma) that
\beq
\label{X5}
\norm{G}_{L^\infty((0,T')^2;V)} \leq  C_2\exp( C_2T' \N e^{C_2T' \N} ) .
\eeq
By \eqref{det G}, each entry of $G^{-1}$ is a cofactor of $G$, so that we infer that 
\beq
\label{X5bis}
\norm{G ^{-1}}_{L^\infty ( (0,T')^2;V)} \leq  C_3 \exp(C_3T' \N  e^{C_3T'\N } ).
\eeq

If $f\in C^{1} ( \R ^3,\R ^3  )$ with $\partial f/\partial y\in C^{0,\alpha} (\R ^3, \R ^3)$ and 
$g\in C^{1,\alpha } (\R ^3, \R ^3)$, then 
\ba
\vert g\circ f \vert _{0, \alpha }&\le & C\vert g\vert _{0,\alpha } \Vert \frac{\partial f }{\partial y}\Vert ^\alpha _{L^\infty } , \label{X21} \\  
\Vert \frac{\partial }{\partial y} (g\circ f) \Vert _{L^\infty } &\le& C \Vert \frac{\partial g}{\partial y} \Vert _{L^\infty } \Vert \frac{\partial f}{\partial y}  \Vert _{L^\infty } , \label{X22}\\
\vert \frac{\partial }{\partial y} (g \circ f ) \vert _{0, \alpha } &\le & C 
\left( \Vert \frac{\partial g}{\partial y} \Vert _{L^\infty}\,  \vert \frac{\partial f}{\partial y} \vert _{ 0, \alpha }   
+ \Vert \frac{\partial f}{\partial y} \Vert ^{1+ \alpha}  _{L^\infty}  \, \vert \frac {\partial g}{\partial y} \vert _{ 0 , \alpha }    \right) .\label{X23} 
\ea

Using \eqref{defG}, \eqref{equation for G},  \eqref{cota_v},   \eqref{X2}, \eqref {X21} and Gronwall's lemma, we obtain that 
\be
\Vert G\Vert _{ L^\infty ((0,T')^2; C^{0,\alpha } (\R ^3 ) ) } \le C \exp (CT' \N  e^{CT'\N} ). \label{X24}
\ee

Next, it follows from \eqref{equation for G}, \eqref{cota_v}, \eqref{X22}, \eqref{X23} and \eqref{X24} that 
\be
\Vert G\Vert _{ L^\infty ((0,T')^2; C^{1,\alpha } (\R ^3 ) ) } \le C_4  \exp (C_4T' \N e^{C_4T'\N} ). \label{X25}
\ee
 
Using again \eqref{det G}, we obtain that
\be
\norm{G ^{-1}}_{L^\infty( (0,T')^2; C^{1, \alpha} (\R ^ 3) )} \leq   C_5 \exp(C_5T' \N e^{C_5 T'\N } ). \label{X26}
\ee

We are in a position to derive the required estimates for $\hat \omega$. 
From \eqref{X21}-\eqref{X23} and \eqref{X25}, we infer that 
\[
\Vert \hat \omega _0( \hat X(0;t,.)) \Vert  _{L^\infty( 0,T'; C^{1, \alpha} (\R ^3 ) )} \leq   C \exp(CT' \N e^{C T'\N } )\Vert \hat \omega _0\Vert _{C^{1,\alpha } (\R ^3)} 
\]
which yields with  \eqref{equation rot01} and \eqref{X26}
\be
\Vert \hat \omega  \Vert  _{L^\infty( 0,T'; C^{1, \alpha} (\R ^3 ) )} 
\leq   C_6 \exp(C_6T' \N e^{C_6 T'\N } )\Vert \hat \omega _0\Vert _{C^{1,\alpha } (\R ^3)}. 
\ee

From \eqref{Y1}-\eqref{Y2} we obtain 
\[
\Vert \hat \omega _0( \hat X(0;t,.))  \Vert  _{L^\infty( 0,T'; M^p_{1,\delta +2} (\R ^3) ) } 
\leq   C \exp ( C T' \N ) \Vert \hat \omega _0\Vert _{M^p_ {1, \delta +2} (\R ^3)  } 
\]
which gives with \eqref{X5bis}
\be
\Vert \hat \omega  \Vert  _{ L^\infty( 0,T'; M^p_{1,\delta +2} (\R ^3) )} 
\leq   C_7 \exp(C_7T' \N e^{C_7 T'\N } )\
 \Vert \hat \omega _0\Vert _{ M^p_{1, \delta +2} (\R ^3)  } .
\ee
  
 If we define $C_8 =\max\{ C_6,C_7 \}$, and take $T' >0$ such that
 \beq \label{cond01}
 T'\leq\frac{1}{C_8\N } ,
 \eeq 
then we obtain 
 \ba
 \|\hat{\omega}\|_{L^\infty(0,T';C^{1,\alpha}(\R^3))} +\|\hat{\omega}\|_{L^\infty(0,T'; M^p_{1, \delta +2 } (\R ^3 )  ) }   
   &\leq&e ^e\cdot  ( C_6  \|\hat{\omega}_0\|_{C^{1,\alpha}(\R^3)} + C_7 \|\hat{\omega}_0\| _{  M^p_{1, \delta +2 } (\R ^3 ) } ) \nonumber \\
  &=:&P.
 \ea
On the other hand, if we consider any scalar function $\varphi\in C^1(\R^3)$ with compact support,  we obtain by using the change of variables $y=\hat X(t;0,x)$
\begin{eqnarray*}
\int_{\R^3}\hat{\omega}(t,y) \cdot  \nabla  \varphi(y)dy
&=& \int_{\R ^3} G^{-1} (0; t,y) \hat \omega _0 (\hat X (0 ; t , y) )   \cdot  \nabla  \varphi (y) dy\\
& =&\D \int_{\R^3}G^{-1}(0;t,\hat X(t;0,x))\hat{\omega}_0(x) \cdot  \nabla_y \varphi(\hat X(t;0,x))dx,\\
&=&\D \int_{\R^3}G(t;0,x)\hat{\omega}_0(x) \cdot  \nabla_y \varphi(\hat X(t;0,x))dx\\
&=&\D \int_{\R^3}\hat{\omega}_0(x) \cdot  \nabla_x [\varphi(\hat X(t;0,x))]dx.
\end{eqnarray*}
Since, by \eqref{X1},
the function $\varphi(\hat X(t;0,\cdot))\in C^1(\R^3)$ has a compact support, we infer from  $\textnormal{div}(\hat\omega_0)=0$ that 
\beq
\textnormal{div}(\hat\omega)=0 \;\;\textnormal{ in } \R^3.
\eeq
 Integrating over $\Omega$, we obtain
  \beq
  \int_\Omega \hat\omega\cdot\n d\sigma=0.
  \eeq
 On the other hand,  $\hat \omega (0) _{\vert \overline{\Omega} } =\omega _0$. 
 Therefore,  the condition about $\hat{\omega}$ for $ ( \hat l , \hat r, \hat \omega )$  to belong to $\mathcal{C}$ is satisfied. \\
 
 {\bf Step 3.} Let us turn our attention to $(\hat{l},\hat{r})$. Since $M^p_{1, \delta + 2} (\R ^3) \subset L^\infty(\R ^3)$ and $\nabla \hat v(t) \in M^p_{1, \delta + 2} (\R ^3)$ for all $t\in [0,T']$, 
 we infer from \eqref{X0} that for all $t\in [0,T']$
\[
\int_{\R ^3} |  \nabla \hat v |^4 \langle y\rangle ^{ p(\delta +2) }  dy  
\le \Vert  \nabla \hat v \Vert_{L^\infty ( \R ^ 3) } ^{4-p} \int_{\R ^3} | \nabla {\hat v} |^p \langle y\rangle ^{p(\delta +2)} dy
\le C \Vert \nabla \hat v\Vert ^4 _{M^p_{1,\delta +2} (\R ^3) }
 < + \infty .
\]
Furthermore,
\be
\Vert \nabla \hat v \Vert _{L^\infty (0,T', L^4 (\R ^ 3 ) ) }
\le C \Vert \nabla \hat v \Vert _{L^\infty (0,T', M^p_{1,\delta +2} (\R ^3)  ) } \le C_9 \N .  
\ee
Using  \eqref{system_mu} and \eqref{cota_v}, we infer  that
\beq
\|\nabla\mu\|_{L^\infty (0,T',L^2(\Omega) )}\leq C_{10} \N ^2.
\eeq
From \eqref{sistema_l_r},  we deduce that $(\hat l , \hat r) \in C([0,T' ];  \R ^6)$ with 
$$\norm{(\hat{l},\hat{r})-(l_0,r_0)}_{L^\infty (0,T') }\leq C_{11}T' \N ^2.$$
On the other hand, $(\hat l (0), \hat r(0) ) = (l_0,r_0)$. 
Therefore, the condition about $(\hat{l},\hat{r})$ for \T$\!\! (l, r, \omega)$ to belong to $\mathcal{C}$ is satisfied provided
that
\beq 
\label{cond02}
T'\leq\frac{N}{C_{11} \N ^2} \cdot
\eeq 
Hence for $T'$ satisfying \eqref{cond01} and \eqref{cond02}, one has \T$\!\!   (\mathcal{C}) \subset \mathcal{C}$.

Note also that, since $\nabla \hat v(t) \in M^p_{1,\delta +2} (\R ^3)$ for all $t\in [0,T]$, we have 
\[
\langle y\rangle \nabla \hat v (t) \in M^p_{1,\delta +1} (\R ^3)  \subset C_0(\R ^3 ) \quad
\textrm{ for all }  t\in [0,T']. 
\]
\\

{\bf Step 4.}
Now, we prove that  the operator  \T is a {\em contraction} in  $ {\mathcal C}$ for the distance induced by $\interleave \cdot \interleave$ for $T'$ small enough. 

From now on, the constant $C$ may depend on $\mathcal N$, but not on $T'$ or on $ (l ^k,r ^k,\omega ^k)$.

Assume given $(l^k,r^k,\omega^k) \in {\mathcal C}$, $k=1,2$.  
Note that $(l^1,r^1,\omega^1)$ and $(l^2,r^2,\omega ^2)$ correspond to the same initial data $(l_0,r_0,\omega _0)$ and the same control input $\omega$. 

Let us introduce for $k=1,2$ 
$$ \textnormal{\T} \!\! ( l^k , r^k , \omega^k ) := ( \hat{l}^k , \hat{r}^k , \hat{\omega}^k ).$$
Then,  for $k = 1,2$, $\hat\omega^k$ fulfills 
\beq\label{equation rot02}
\hat{\omega}^k(t,y)=A^k(0;t,y)\hat{\omega}_0(\hat{X}^k(0;t,y)),
\eeq
where $\hat{\omega}_0=\pi(w_0)$, $\hat{X}^k$ denotes the solution to
\beq\label{flux of v tilde 02}
\left\{
\begin{array}{ccl}
\D\frac{\partial}{\partial s} \hat{X}^k(s;t,y)&=&
\hat{v}^k(s,\hat{X}^k(s;t,y))-l^k(s)-r^k(s)\times\hat{X}^k(s;t,y),\\\\
\hat{X}^k(t;t,y)&=& y,

\end{array}\right.
\eeq
\beq\label{defG 02}
 G^k(s;t,y) : =\frac{\partial\hat{X}^k}{\partial y}(s;t,y),\;\;\;G(t;t,y)=Id ,
 \eeq
 and $A^k : =(G^k)^{-1}$. 
The velocity $\hat{v}^k=\pi(v^k) :\R ^3\to \R^3$ is the extension of the velocity $v^k:\Omega \to \R^3$ decomposed as
$$
v^k=\eta^k+\sum\limits_{i=1}^3l^k_i\nabla \phi_i+\sum\limits_{i=1}^3r^k_i\nabla \varphi_i
+\sum\limits_{1\leq j\leq m}w_j(t)\nabla\psi_j,
$$
where $\eta^k$ is the solution of 
\ba
\textnormal{curl}\,\eta^k &=& \omega^k, \quad  (t,y) \in (0,T)\times\Omega,\\
\textnormal{div}\,\eta^k &=& 0,\  \  \quad   (t,y) \in (0,T)\times\Omega, \\
\D \eta^k \cdot \n &= & 0, \ \ \quad  (t,y) \in (0,T)\times\partial\Omega, \\
\lim\limits_{\vert  y\vert \to+\infty}\eta^k(t,y) &=& 0,\ \ \  \quad  t\in (0,T) .
\ea

We introduce the functions
\beq
v:=v^1-v^2,\;\;\eta :=\eta^1-\eta^2,\;\;r : =r^1-r^2,\;\; l: =l^1-l^2, \;\; \omega : =\omega^1-\omega^2, \textnormal{ and } A: =A^1-A^2.
\eeq
Thus $v$ may be written as 
\beq\label{eq_velocity01}
v=\eta+\sum\limits_{i=1}^3l_i\nabla \phi_i+\sum\limits_{i=1}^3r_i\nabla \varphi_i,\;\;v(0,y)=0,
\eeq
where $\eta$ is the solution to the system  
\ba
\textnormal{curl}\,\eta   &=&  \omega , \quad   (t,y) \in (0,T)\times\Omega, \\
\textnormal{div}\,\eta   &=&  0, \  \quad   (t,y) \in (0,T)\times\Omega, \\
\D \eta \cdot \n  &=& 0, \ \quad (t,y) \in (0,T)\times\partial\Omega, \\
\lim\limits_{\vert y \vert \to+\infty}\eta(t,y) &=& 0,\  \  \quad  t\in (0,T). 
\ea

{\bf Step 5.} 
Let $\hat{X} : =\hat{X}^1-\hat{X}^2.$ Then
\beq\label{flux of v tilde 03}
\left\{
\begin{array}{ccl}
\D\frac{\partial}{\partial s} \hat{X}(s;t,y)&=&\hat{v}^1(s,\hat{X}^1(s;t,y))-\hat{v}^1(s,\hat{X}^2(s;t,y))+\hat{v}(s,\hat{X}^2(s;t,y))\\
& &-l(s)-r^1(s)\times\hat{X}(s;t,y)-r(s)\times\hat{X}^2(s;t,y),\\\\
\hat{X}(t;t,y)&=& 0,
\end{array}\right.
\eeq
where 
$\hat{v} :=\hat{v}^1-\hat{v}^2 = \pi (v)$.

Taking the scalar product of each term  in \eqref{flux of v tilde 03}
 by $\hat X$ results in 
\[
\vert \hat X\vert \frac{\partial \vert \hat X \vert  }{\partial s} =\frac{\partial}{\partial s} (\frac{1}{2} \vert \hat X\vert ^2) 
= \hat X \cdot  \frac{\partial \hat X}{\partial s } = 
\big( \hat v^1 (s, \hat X^1) -\hat v^1 (s, \hat X^2) + \hat v (s,\hat X^2) -l-r\times \hat X^2 \big) \cdot \hat X.
\]
It follows that 
\[
\frac{\partial \vert \hat X \vert  }{\partial s}  \le C
\left(  \Vert \frac{\partial \hat  v^1}{\partial y} \Vert _{L^\infty (\R ^3)} \vert \hat X\vert 
+ \Vert \hat v\Vert_{L^\infty  ( \R ^3) } + \vert l \vert + \vert r\vert \cdot   \vert \hat  X^2\vert \right) .  
\]
Since 
\[
\Vert \hat v (s) \Vert _{L^\infty (\R ^3)} 
\le C  \Vert v (s) \Vert_{ C^{0, \alpha}  (\overline{\Omega} ) } 
\le C \left( \Vert \eta (s)  \Vert_{ C^{0, \alpha}  (\overline{\Omega} )}  + \vert l(s) \vert  + \vert r(s)\vert  \right) 
\le C\interleave (l,r,\omega )  \interleave 
\]
and 
\[
\vert \hat X^2\vert \le  | y |  + CT' \N  \le C \langle y \rangle ,
\] 
we obtain with Gronwall's lemma that for $(s,t,y)\in [0,T'] ^2\times \R ^3$, 
\begin{equation}
\vert \hat X (s,t,y) \vert  \le  e^{CT' } \int_0^{T'} C(1 + \langle y\rangle) \interleave (l,r,\omega )  \interleave dt\\
\le  CT' \langle y\rangle \interleave (l,r,\omega ) \interleave . \label{ppp1}
\end{equation}


{\bf Step 6.} 
Let us 
set $A:=A^1-A^2$. (Recall that $A^k=(G^k)^{-1}$ for $k=1,2$.) Then  we notice that 
$$\frac{\partial A^k}{\partial s}(s;t,y)=  -  A^k(s;t,y)
\Big(\frac{\partial\hat{v}^k}{\partial y}(s,\hat{X}^k(s;t,y))-S(r^k(s))\Big), \;\;\; A^k (t;t,y)=Id.$$
Thus 
\ba
\frac{\partial A}{\partial s}(s;t,y)&=& -A(s;t,y)\Big(\frac{\partial\hat{v}^1}{\partial y}(s,\hat{X}^1(s;t,y))-S(r^1(s))\Big)\nonumber\\
&&\quad -A^2(s;t,y)\Big(\frac{\partial\hat{v}^1}{\partial y}(s,\hat{X}^1(s;t,y))-\frac{\partial\hat{v}^1}{\partial y}(s,\hat{X}^2(s;t,y))\Big)\nonumber\\
&&\quad -A^2(s;t,y)\Big(\frac{\partial\hat{v}}{\partial y}(s,\hat{X}^2(s;t,y))-S(r(s))\Big),
\label{uuu1}\\
A(t;t,y) &=& 0. \label{uuu2}
\ea
It follows that 
\be
\label{uvw}
\Vert \frac{\partial A}{\partial s}(s;t,y) \Vert 
\le 
C\left( \Vert A\Vert \big( \Vert \frac{\partial \hat v^1}{\partial y} \Vert _{ L^\infty (\R ^3  )}   + \vert r^1\vert \big) + 
\Vert A^2\Vert \cdot \Vert \frac{\partial \hat v^1}{\partial y} \Vert _{W^{1,\infty} (\R ^3 ) } \vert \hat X\vert 
+\Vert A ^2\Vert  \big( \Vert \frac{\partial \hat  v}{\partial y} \Vert _{L^\infty (\R ^3 )} + \vert r\vert \big) \right) .  
\ee
From (\ref{cota_v}), we have that  
\[
\Vert \frac{\partial \hat v^1}{\partial y} (s)  \Vert _{ W^{1,\infty} (\R ^3 )}
\le \Vert \hat v^1(s)  \Vert _{ C^{2,\alpha } (\R ^3 )}
 \le C\N \le C. 
\]
Clearly, 
\[
\Vert A^2 \Vert + \vert r^1\vert \le C\N \le C.  
\]
On the other hand, it follows from Morrey's inequality that 
\[ \Vert \eta \Vert _{C^{0, 1-\frac{3}{p}} ( \overline{\Omega} ) } \le 
C \Vert \eta \Vert _{M^p_{1, \delta +1}} \le C \Vert \omega \Vert _{M^p_{0,\delta +2}} .\]
Let 
\[
\vert f\vert _{0,\alpha , \R ^3}  := \sup  \{ \frac{\vert  f(x)-f(y) \vert  }{|x-y|^\alpha } ; \ x,y\in \R^3, \ x\ne y\} .
\]
Then,  since $0<\alpha \le  1-\frac{3}{p}$, we have
\begin{eqnarray}
\Vert \frac{\partial \hat v}{\partial y} (s)  \Vert _{L^\infty (\R ^3 ) }  + \vert \frac{\partial \hat v}{\partial y} (s)  \vert _{0, \alpha  ,\R ^3} 
&\le& C \Vert v (s)  \Vert _{C^{1,\alpha }   (\overline {\Omega} ) } \nonumber \\
&\le& C \left(   \Vert \eta (s)  \Vert _{ C^{1,\alpha }  ( \overline{ \Omega} ) }   + \vert l(s)\vert + \vert r(s)\vert  \right) \nonumber \\ 
&\le& C \left(   \Vert \eta (s) \Vert _{C^{0, \alpha } (\overline{\Omega} ) } 
+ \Vert \omega  (s) \Vert _{C^{0,\alpha }  (\overline{\Omega} )}   + \vert l(s)\vert + \vert r(s)\vert  \right) \nonumber\\ 
&\le& C \left(   \Vert \omega  (s) \Vert _{M^p_{0, \delta + 2} } 
+ \Vert \omega (s)  \Vert _{C^{0,\alpha }  (\overline{\Omega} )}   + \vert l(s)\vert + \vert r(s)\vert  \right) \nonumber\\
&\le&  C \interleave (l,r,\omega )  \interleave . \label{qqq1}
\end{eqnarray}
We infer with (\ref{ppp1}) and \eqref{uvw} that 
\[
\Vert \frac{\partial A}{\partial s}(s;t,y) \Vert 
\le 
C \Vert A\Vert  
+  C(T' \langle y\rangle  + 1)  \interleave (l,r,\omega )  \interleave .
\]

Since $A(t ; t , y)=0$, we obtain by using Gronwall's lemma that for $(s,t,y)\in [0,T']^2\times \R ^3$ 
\beq
\Vert A(s;t,y)\Vert \le CT' \langle y\rangle \interleave (l,r,\omega )  \interleave . 
\label{ppp2}
\eeq
\\

{\bf Step 7.} 
Let $\hat\omega : =\hat\omega^1-\hat\omega^2$.
We first give an  estimate  of $\Vert \hat \omega\Vert _{ L^\infty  (\Omega ) } $. We write  
\ba
|\hat\omega| &=& 
\D \left|A^1(0;t,y)\hat{\omega }_0(\hat{X}^1(0;t,y))-A^2(0;t,y)\hat{\omega}_0(\hat{X}^2(0;t,y))\right| \nonumber\\
&\leq& \D\left|A(0;t,y)\hat{\omega }_0(\hat{X}^1(0;t,y))\right|+\left|A^2(0;t,y)\left(\hat{\omega }_0(\hat{X}^1(0;t,y))-\hat{\omega }_0(\hat{X}^2(0;t,y))\right)\right|.
\ea
Since $\omega _0 \in M^p_{1, \delta + 2}$, we have by \cite[Lemma 2.2]{kikuchi86} that 
\be 
\label{ppp50}
\vert \hat \omega _0 (y) \vert  =O(\vert y\vert ^{-\delta -2})\  \textrm{ as } \ \vert y\vert \to + \infty ,
\ee
 so that
we infer from \eqref{X1} and  \eqref{ppp2} that 
\[
\left| A(0;t,y)\hat{\omega }_0(\hat{X}^1(0;t,y)) \right| \le CT' \langle y\rangle \interleave (l,r,\omega )  \interleave  \vert \hat{\omega }_0(\hat{X}^1(0;t,y) \vert \le CT' \interleave (l,r,\omega )  \interleave .
\] 
On the other hand, by  (\ref{Homega}), (\ref{X26}) and (\ref{ppp1}), we have that 
\[
\left|A^2(0;t,y)\left(\hat{\omega }_0(\hat{X}^1(0;t,y))-\hat{\omega }_0(\hat{X}^2(0;t,y))\right)\right|
\le \frac{C}{1 + \min (\vert \hat X^1\vert , \vert \hat X^2\vert )}\vert \hat X\vert  \le CT' \interleave (l,r,\omega )  \interleave ,  
\] 
where we used (\ref{X1}) to get $1+ \min (\vert \hat X^1\vert , \vert \hat X^2\vert ) \ge C\langle y\rangle$ for $y\in \Omega$ 
and $t\in [0,T']$.  
Thus, we have proved that for $T' >0$ satisfying \eqref{cond01} and \eqref{cond02}, we have 
\begin{equation}
\label{ppp3}
\Vert \hat \omega \Vert _{L^\infty (0,T'; L^\infty (\Omega ))}  \le CT' \interleave (l,r,\omega )  \interleave .
\end{equation}
\\

{\bf Step 8.} 
Let us now estimate the H\"older norm $\vert \hat \omega\vert_{0,\alpha }$.  Note first that it is not clear whether $\hat X \in C^{0,\alpha}(\overline{\Omega})$, since
it could happen that  $\hat X \sim \langle y\rangle$  as $\vert y\vert \to +\infty$ (and hence, $\hat X\not\in L^\infty (\Omega )$).  Rather, we shall prove that $\langle y\rangle ^{-1} \hat X\in C^{0, \alpha } (\overline{\Omega} )$. 

We infer from \eqref{flux of v tilde 03} that 
\begin{eqnarray*}
\D\frac{\partial}{\partial s} (\langle y\rangle ^{-1} \hat{X} (s;t,y) )&=& \langle y\rangle ^{-1} \left( \hat{v}^1(s,\hat{X}^1(s;t,y))-\hat{v}^1(s,\hat{X}^2(s;t,y))+\hat{v}(s,\hat{X}^2(s;t,y)) \right. \\
&  &\quad  \left. -l(s)-r^1(s)\times\hat{X}(s;t,y)-r(s)\times\hat{X}^2(s;t,y) \right) , \\
&=& \int _0^1 \frac{\partial \hat v^1} {\partial y} (s, \hat X^2 +\sigma \hat X) \langle y\rangle ^{-1} \hat X d\sigma \\
&& \quad  + \langle y\rangle ^{-1} (\hat v (s, \hat X^2) -l -r^1\times \hat X -r\times \hat X^2)\\ 
\end{eqnarray*}
Therefore, using \eqref{X21}, we obtain
\begin{eqnarray}
\vert \D\frac{\partial}{\partial s} ( \frac{ \hat{X} }{\langle y\rangle } ) \vert _{0, \alpha , \R ^3} &\le &
[ \Vert \frac{\partial \hat v^1 } {\partial y}\Vert _{L^\infty (\R ^3)} \vert \frac{ \hat X}{\langle y\rangle } \vert _{0,\alpha , \R ^3} 
+  \vert \frac{\partial \hat v^1 } {\partial y}\vert _{0, \alpha , \R ^3} \big( \Vert \frac{\partial \hat X^2}{\partial y} \Vert _{L^\infty (\R ^3 )} 
+ \Vert \frac{\partial \hat X^1}{\partial y} \Vert _{L^\infty (\R ^3 )} )^\alpha  
\Vert \frac{ \hat X}{\langle y\rangle } \Vert _{L^\infty (\R ^3)} ] \nonumber \\ 
&& + \vert \frac{  \hat v (s, \hat X^2 )}{\langle y\rangle } \vert _{0, \alpha, \R ^3}  + C  (1 + \vert \frac{ \hat X ^2}{\langle y\rangle } \vert _{0,\alpha, \R ^3 } ) \interleave (l,r,\omega )  \interleave 
+ C \vert \frac{ \hat X}{\langle y\rangle } \vert _{0,\alpha } . \label{ppp5}
\end{eqnarray}
It is clear that 
\be
 \Vert \frac{\partial \hat v^1 } {\partial y}\Vert _{L^\infty (\R ^3)} 
+  \vert \frac{\partial \hat v^1 } {\partial y}\vert _{0, \alpha  ,\R ^3 } +  \Vert \frac{\partial \hat X^2}{\partial y} \Vert _{L^\infty (\R ^3 )} 
+ \Vert \frac{\partial \hat X^1}{\partial y} \Vert _{L^\infty (\R ^3 )}  
\le C. \label{ppp6}
\ee
To bound  $ \vert \frac{ \hat X ^2}{\langle y\rangle }  \vert _{0,\alpha , \R ^3}$, we notice that 
$\frac{ \hat X ^2}{\langle y\rangle } $ solves the system 
\begin{eqnarray*}
\frac{\partial}{\partial s} (\frac{ \hat X ^2}{\langle y\rangle } ) &=& \frac{\hat v ^2 (s, \hat X^2 )}{\langle y\rangle } -\frac{l^2}{\langle y\rangle} 
-r^2 \times \frac{\hat X^2}{\langle y\rangle } ,\\
\frac{ \hat X ^2}{\langle y\rangle } (t;t,y) &=& \frac{y}{\langle y\rangle } \cdot
\end{eqnarray*}
Since $\vert \hat v^2 (s, \hat X^2 )\vert _{0,\alpha ,\R ^3 }  \le C  \vert \hat v^2\vert _{0, \alpha } \Vert \frac{\partial \hat X^2}{\partial y}\Vert ^\alpha _{ L^\infty (\R ^3 )} \le C$,   
$\Vert \hat v^2 (s, \hat X^2 )\Vert _{L^\infty (\Omega ) } \le C$,  
and both $\langle y\rangle ^{-1}$ and $\langle y\rangle ^{-1} y$ belong to $W^{1,\infty}(\R ^3) \subset C^{0,\alpha } (\R ^3)$, we obtain with Gronwall's lemma that 
\be
\vert \frac{ \hat X ^2}{\langle y\rangle } \vert _{0, \alpha , \R ^3 } \le C. 
\label{ppp7}
\ee
On the other hand, we infer from \eqref{X21} and \eqref{qqq1} that
\[
 \vert \frac{  \hat v (s, \hat X^2 )}{\langle y\rangle } \vert _{0, \alpha, \R ^3} \le C \left( \Vert \hat v (s, X^2)\Vert _{L^\infty (\R ^3)}  + 
  \vert  \hat v (s, \hat X^2 ) \vert _{0, \alpha, \R ^3} \right)  \le C  \interleave (l,r,\omega )  \interleave .
\]
It follows from \eqref{flux of v tilde 03}, \eqref{ppp5}-\eqref{ppp7} and Gronwall's lemma that 
\be
 \vert  \frac{ \hat{X} }{\langle y\rangle } \vert _{0, \alpha} \le CT'  \interleave (l,r,\omega )  \interleave .
\label{ppp8}
\ee

Next,  we prove that a similar estimate holds for $\vert \frac{A}{\langle y\rangle } \vert _{0, \alpha }$. 
Writing for $1\le i,j\le 3$
\[
\frac{\partial\hat{v}^1_i}{\partial y_j}(s,\hat{X}^1)-\frac{\partial\hat{v}^1_i}{\partial y_j}(s,\hat{X}^2) 
=\int_0^1 \frac{\partial }{\partial y} (\frac{\partial\hat{v}^1_i}{\partial y_j}) (s,\hat{X}^2 +\sigma \hat X ) \cdot \hat X d\sigma,
\] 
and using \eqref{uuu1},  we infer that 
\begin{eqnarray*}
\vert \frac{\partial }{\partial s} ( \frac{A}{\langle y\rangle } )  \vert _{0, \alpha } 
&\le & C\  \vert \frac{A}{\langle y\rangle }\vert _{0, \alpha } \left(  \Vert \frac{\partial \hat v^1}{\partial y }\Vert _{L^\infty (\Omega )}  + \vert r^1\vert \right) 
 +   \, C\, \Vert \frac{A}{\langle y\rangle}  \Vert _{L^\infty (\Omega )} \vert   \frac{\partial \hat v^1 }{\partial y } \vert _{0,\alpha} \Vert \frac{\partial \hat X^1}{\partial y} 
\Vert ^\alpha_{L^\infty (\Omega )}  \\
&& +  \, C\, \Vert A^2\Vert _{L^\infty (\Omega )} \left(  \vert  \hat v_1 \vert _{2,\alpha } ( \Vert \frac{\partial \hat X^2}{\partial y} \Vert _{L^\infty (\Omega )} 
+ \Vert \frac{\partial \hat X^1}{\partial y}\Vert _{L^\infty (\Omega )})^\alpha  \Vert \frac{\hat X}{\langle y\rangle } \Vert _{L^\infty (\Omega)}  \right. \\
&&\qquad \qquad  \qquad \quad \left. +\Vert \frac{\partial }{\partial y} (\frac{\partial\hat{v}^1}{\partial y}) \Vert _{L^\infty (\Omega)}  \vert \frac{\hat X}{\langle y\rangle }\vert _{0,\alpha}  \right) \\
&& +\, C\, \vert A^2\vert _{0,\alpha } \Vert \hat v^1 \Vert _{W^{2,\infty}(\Omega )} \Vert\frac{\hat X}{\langle y \rangle}  \Vert_{L^\infty (\Omega)} \\  
&&  + \, C\,  \vert \frac{A^2}{\langle y\rangle }\vert _{0, \alpha } \left(  \Vert \frac{\partial \hat v}{\partial y }\Vert _{L^\infty (\Omega )}  + \vert r\vert \right) 
 +   C\Vert \frac{A^2}{\langle y\rangle}  \Vert _{L^\infty (\Omega )} \vert   \frac{\partial \hat v}{\partial y } \vert _{0,\alpha} \Vert \frac{\partial \hat X^2}{\partial y} 
\Vert ^\alpha_{L^\infty (\Omega )}. 
\end{eqnarray*}
Then using  \eqref{ppp1}, \eqref{qqq1}, \eqref{ppp2}, we infer that 
\[
\vert \frac{\partial }{\partial s} ( \frac{A}{\langle y\rangle } )  \vert _{0, \alpha } \le C\vert \frac{A}{\langle y\rangle }\vert _{0, \alpha } + C(1+T')  \interleave (l,r,\omega )  \interleave .
\]
Therefore, using the fact that $A(t;t,y)=0$, we obtain with Gronwall's lemma that 
\be
\label{ppp10}
\vert  \frac{A}{\langle y\rangle }   \vert _{0, \alpha } \le C T' \interleave (l,r,\omega )  \interleave .
\ee
We are in a position to estimate $\vert \hat \omega \vert _{0,\alpha}$. We have 
\begin{eqnarray*}
\vert \hat\omega \vert _{0,\alpha } &\le & 
\left\vert A(0;t,y) \hat \omega _0(\hat X^1(0;t,y))\right\vert _{0, \alpha }  + \left\vert A^2(0;t,y)\left( \hat \omega _0(\hat X^1(0;t,y)) -\hat \omega _0 (\hat X ^2 (0;t,y)) \right) \right\vert  _{0, \alpha } \\
&\le& \left\vert \frac{A}{\langle y\rangle } \right\vert _{0, \alpha } \Vert \langle y\rangle \hat \omega _0 (\hat X^1)  \Vert _{L^\infty (\Omega ) } 
+ \Vert \frac{A}{\langle y \rangle }  \Vert _{L^\infty (\Omega )} \vert \langle y\rangle \hat \omega _0 (\hat X^1)\vert _{0,\alpha}  \\
&& + \vert A^2 \vert _{0, \alpha } \Vert \hat \omega _0 (\hat X^1) -\hat \omega _0 (\hat X^2) \Vert _{L^\infty (\Omega )} 
+ \Vert A^2 \Vert _{L^\infty (\Omega )} \vert \hat \omega _0 (\hat X^1) -  \hat \omega _0 ( \hat X^2 )\vert _{0, \alpha } \\
&\le& C(1+  \vert \langle y\rangle \hat \omega _0 (\hat X^1)\vert _{0,\alpha} ) T' \interleave (l,r,\omega )  \interleave  
+ C \vert \hat \omega _0 (\hat X^1) -\hat \omega _0 (\hat X^2) \vert _{0, \alpha} 
\end{eqnarray*} 
where we used \eqref{Homega}, \eqref{ppp1}, \eqref{ppp2},  and  \eqref{ppp10}.  It remains to estimate
$\vert \hat \omega _0 (\hat X^1) - \hat \omega _0 (\hat X^2) \vert _{0, \alpha}$  and  $\vert \langle y\rangle \hat \omega _0 (\hat X^1)\vert _{0,\alpha}$. 
For the first one,  we write 
\begin{eqnarray*}
\vert \hat \omega _0 (\hat X^1) - \hat \omega _0 (\hat X^2) \vert _{0, \alpha} 
&=& \left\vert \int_0^1 \frac{\partial \hat \omega _0}{\partial y} (\hat X^2 + \sigma \hat X) \hat X d \sigma \right \vert _{0, \alpha } \\
&\le& C \left( \sup_{\sigma \in (0,1)} \Vert \langle y\rangle \frac{\partial \hat  \omega _0}{\partial y} (\hat X^2 +\sigma \hat X) \Vert _{L^\infty (\Omega )}  \left\vert  \frac{\hat X}{\langle y\rangle } \right\vert _{0,\alpha} \right.\\
&&\qquad \left. + \sup_{\sigma \in (0,1)} 
\left\vert  \langle y\rangle  \frac{\partial \hat\omega _0}{\partial y} (\hat X^2 + \sigma \hat X) \right\vert _{0, \alpha} \Vert \frac{\hat X}{\langle y\rangle }\Vert _{L^\infty (\Omega )} 
\right)    \\
&\le & C\left( 1+  \sup_{\sigma \in (0,1)}  \left\vert  \langle y\rangle  \frac{\partial \hat\omega _0}{\partial y} (\hat X^2 + \sigma \hat X) \right\vert _{0, \alpha} \right) T' \interleave (l,r,\omega )  \interleave 
\end{eqnarray*}
where we used \eqref{zzz1},  \eqref{zzz2},  \eqref{ppp1},  \eqref{ppp8}, and the fact that  (using \eqref{X1} for $\hat X^2$)  
\[
\vert \hat X^2 +\sigma \hat X \vert \ge \vert y\vert +O(T')\langle y \rangle  \ge \frac{1}{2} \langle y\rangle  \textrm{ for } T' \textrm{ small enough} , \   \vert y\vert >1  \textrm{ and }
\sigma \in (0,1).
\] 

We aim to prove that 
\be
\label{ppp20}
\sup _{ \small \begin{array}{c}\sigma \in (0,1)\\ t\in [0,T'] \end{array}}
\left\vert  \langle y\rangle  \frac{\partial \hat\omega _0}{\partial y} 
(\hat X^2 (y)  + \sigma \hat X (y) ) \right\vert _{0, \alpha} < + \infty ,
\ee 
where we write $\hat X(y)$ for $\hat X(0 ; t , y)$, etc. 
We have with \eqref{zzz1} that 
\[
\sup _{\small \begin{array}{c}\sigma \in (0,1)\\ \vert y -y' \vert \ge 1 \\
t\in [0,T']
 \end{array}} \frac{ \left\vert  \langle y\rangle  \frac{\partial \hat\omega _0}{\partial y} (\hat X^2 (y)  + \sigma \hat X (y) )
- \langle y' \rangle  \frac{\partial \hat\omega _0}{\partial y} (\hat X^2(y') + \sigma \hat X(y'))\right\vert }{\vert y- y' \vert ^\alpha}  < + \infty.  
\]
On the other hand, for $\sigma\in (0,1)$, $\vert y-y'\vert <1$, and $t\in [0,T']$,  
\begin{eqnarray*}
&& \frac{ \left\vert  \langle y\rangle  \frac{\partial \hat\omega _0}{\partial y} (\hat X^2 (y) + \sigma \hat X (y) )
- \langle y' \rangle  \frac{\partial \hat\omega _0}{\partial y} (\hat X^2(y') + \sigma \hat X(y'))\right\vert }{\vert y- y' \vert ^\alpha}\\
 &&\quad \le \left\vert  \frac{\langle y\rangle - \langle y'\rangle }{\vert y-y'\vert ^\alpha}  
 \frac{\partial \hat \omega _0}{\partial y}  (\hat X^2 (y) + \sigma \hat X (y) ) \right\vert 
 +\langle y'\rangle  \frac{ \left\vert   \frac{\partial \hat\omega _0}{\partial y} (\hat X^2 ( y) + \sigma \hat X (y))
-  \frac{\partial \hat\omega _0}{\partial y} (\hat X^2(y') + \sigma \hat X(y'))\right\vert }{\vert y- y' \vert ^\alpha}\\
&&\quad \le C\Vert \frac{\partial \hat\omega _0}{\partial y}\Vert _{L^\infty (\Omega )} 
+ C \big( \Vert \frac{\partial \hat X^2}{\partial y} \Vert _{L^\infty (\Omega )}  + \Vert \frac{\partial  \hat X}{\partial y} \Vert _{L^\infty (\Omega )} \big) \vert y-y'\vert ^{1-\alpha}  
\end{eqnarray*}
where we used \eqref{zzz2} and the mean value theorem for the last term. 
This completes the proof of \eqref{ppp20}. We infer that 
\[
\vert \hat \omega _0 (\hat X^1) - \hat \omega _0 (\hat X^2) \vert _{0, \alpha} \le CT' \interleave (l,r,\omega )  \interleave . 
\]
We can prove in a very similar way that 
$\vert  \langle y\rangle  \hat\omega _0 (\hat X^1) \vert _{0, \alpha} < + \infty$. 
We conclude that 
\be
\vert \hat \omega (t) \vert _{0, \alpha} \le CT' \interleave (l,r,\omega )  \interleave  , \quad t\in [0,T']. 
\label{fff1}
\ee
\\

{\bf Step 9.} 
Let us estimate $\Vert \hat \omega \Vert _{L^p_{p(\delta + 2)} (\Omega )}$. We write
\[
\Vert \hat \omega \Vert _{L^p_{p (\delta +2) } (\Omega )} ^p 
\le
C\left(  \int _\Omega \vert A \hat \omega _0 (\hat X^1) \vert  ^p  \langle y\rangle ^{ p (\delta +2 )} dy + 
\int _\Omega |A^2 (\hat \omega _0 (\hat X^1) -\hat \omega _0 (\hat X^2)) \vert ^p \langle y \rangle ^{p(\delta +2) }dy \right)  =: C( I_1+I_2),   
\]  
where we have written $A^1$ for $A^1(0;t,y)$, $\hat X^1$ for $\hat X^1(0;t,y)$, etc. 

Then, using the fact that $\omega _0\in M^0_{p, \delta + 3} $ and \eqref{ppp2}, we obtain that 
\begin{eqnarray*}
I_1 &\le & ( CT' \interleave (l,r,\omega )\interleave)^p \int_\Omega \vert \hat \omega _0 (\hat X^1 (0;t,y) \vert ^p \langle y\rangle ^{p(\delta +3)}dy \\
&\le& ( CT' \interleave (l,r,\omega )\interleave)^p \int_{\R ^3}  \vert \hat \omega _0 (x) \vert ^p \langle \hat X ^1 (t;0,x) \rangle ^{p(\delta +3)}dx \\
&\le& ( CT' \interleave (l,r,\omega )\interleave)^p \Vert \hat \omega _0 \Vert ^p_{M^p_{0,\delta +3} (\R ^3)}.
\end{eqnarray*}
Therefore, increasing  the value of $C$ if needed, we obtain  
\[
I_1 \le   (CT' \interleave (l,r,\omega )\interleave )^p.
\]

For $I_2$, we infer from \eqref{Homega} with $\kappa > 3+ \delta + \frac{3}{p}$ that
\begin{eqnarray*}
I_2 
&\le& C\int _\Omega \vert \hat \omega _0 (\hat X^1) -\hat \omega _0 (\hat X^2) \vert ^p \langle y \rangle ^{p(\delta +2) }dy \\
&\le& C \int _\Omega \left( \frac{\vert \hat X\vert }{ [1+ \min(\vert \hat X^1\vert, \vert \hat X^2\vert ) ] ^\kappa } \right) ^p    \langle y \rangle ^{p(\delta +2) }dy \\
&\le& ( CT' \interleave (l,r,\omega )\interleave)^p \int_{\R ^3} \langle y\rangle ^{p(\delta + 3 -\kappa )} dy\\
&\le& (CT' \interleave (l,r, \omega )\interleave )^p.
\end{eqnarray*}
We conclude that 
\be
\Vert \hat \omega (t)  \Vert _{L^p_{p(\delta +2) } (\Omega )} \le  CT' \interleave (l,r,\omega )\interleave , \quad t\in [0,T'].
\label{fff2}
\ee
\\

{\bf Step 10.} 
Let us proceed to the estimates of $v$.  Since $\alpha \le 1-\frac{3}{p}$, we have $M^p_{1, \delta + 1}  \subset C^{0,\alpha} (\overline{\Omega} )$, and hence
\begin{eqnarray*}
\Vert \nabla v\Vert _{C^{0,\alpha } (\overline{\Omega} )}  &\le & C ( \Vert \omega \Vert _{C^{0,\alpha} (\overline{\Omega} ) } +
 \Vert v\Vert _{C^{0,\alpha} (\overline{\Omega}) } + \vert l\vert + \vert r\vert )  \\
&\le &   C(  \Vert \omega \Vert _{C^{0,\alpha} (\overline{\Omega }) } + \Vert v\Vert _{M^p_{1, \delta + 1} } + \vert l\vert + \vert r\vert )  \\
&\le& C(  \Vert \omega \Vert _{C^{0,\alpha} (\overline{\Omega }) } + \Vert \omega \Vert _{M^p_{0, \delta + 2} } + \vert l\vert + \vert r\vert )  \\
&\le&  CT' \interleave (l,r,\omega )\interleave .
\end{eqnarray*}
It follows that 
\begin{eqnarray*}
\Vert \nabla v\Vert _{L^4 (\Omega ) } &\le& \Vert \nabla v\Vert ^\frac{p}{4} _{L^p(\Omega  )} \Vert \nabla v\Vert _{L^\infty (\Omega ) } ^{1-\frac{p}{4}} \\
&\le& C\Vert v\Vert _{M^p_{1, \delta +1}} ^\frac{p}{4} \Vert \nabla v\Vert_{C^{0, \alpha} (\overline{\Omega }) } ^{1-\frac{p}{4}}\\
&\le& C(  \Vert \omega \Vert _{C^{0,\alpha} (\overline{\Omega }) } + \Vert \omega \Vert _{M^p_{0, \delta + 2} } + \vert l\vert + \vert r\vert )  \\
&\le&  CT' \interleave (l,r,\omega )\interleave .
\end{eqnarray*}
\\

{\bf Step 11.} 
We now turn our attention to $\hat{l} : = \hat{l}^1-\hat{l}^2$ and $\hat{r} : =\hat{r}^1-\hat{r}^2,$ where for $k=1,2$ 
\beq\label{bbb1}
\left(\begin{array}{c}
\hat{l}^k(t) \\
\hat{r}^k(t)
\end{array}\right)
:=\left(\begin{array}{c}
l_0 \\
r_0
\end{array}\right)
+
\mathcal{J}^{-1}\int_0^t\left\{
\left(\begin{array}{c}
\Big(\int\limits_{\Omega} \nabla \mu^k(\tau,y)\cdot \nabla \phi_i (y) \,d y\Big)_{i=1,2,3}
\\
\Big(\int\limits_{\Omega} \nabla \mu^k(\tau,y)\cdot \nabla \varphi_i(y) \,d y\Big)_{i=1,2,3}
\end{array}\right) -
\left(\begin{array}{c}
m_0r^k\times l^k
\\
r^k\times J_0r^k
\end{array}\right)
\right\}d\tau,
\eeq
and the function $\mu^k : [0,T']\times \Omega  \rightarrow\R$ is defined as the solution to the system
\[
\left\{\begin{array}{ccll}
-\Delta \mu^k&=&\D\textnormal{tr\,}\Big(\nabla v^k\cdot \nabla v^k\Big),&\textnormal{ in }(0,T)\times\Omega, \\
\D\frac{\partial \mu^k}{\partial n}&=&\D-\!\!\!\sum\limits_{1\leq j\leq m}\dot{w}_j(t)\chi_j(y)
-\left(\Big((v^k-l^k-r^k\times y)\cdot\nabla\Big)v^k+r^k\times v^k\right)\cdot \n,&\textnormal{ on }(0,T)\times\partial\Omega ,\\ \\
\lim\limits_{|y|\to\infty}\nabla\mu^k(t,y)&=& 0 &\textnormal{ in }(0,T). 
\end{array}\right.
\]

Then  $\mu :=\mu^1-\mu^2$ satisfies the system
\[
\left\{\begin{array}{ccll}
-\Delta \mu&=&\D\textnormal{tr\,}\Big(\nabla (v^1+v^2)\cdot \nabla v\Big),&\textnormal{ in }(0,T)\times\Omega, \\
\D\frac{\partial \mu}{\partial n}&=&\D
-\left(\Big((v-l-r\times y)\cdot\nabla\Big)v^1+r\times v^1\right)\cdot \n&\\ \\
& &\D-\left(\Big((v^2-l^2-r^2\times y)\cdot\nabla\Big)v+r^2\times v\right)\cdot \n&
\textnormal{ on }(0,T)\times\partial\Omega ,\\ \\
 \lim\limits_{|y|\to\infty}\nabla\mu(t,y)&=&0 & \textnormal{ in }(0,T).
\end{array}\right.
\]
It follows that 
\begin{eqnarray}
\Vert \nabla \mu\Vert _{L^2 (\Omega ) } 
&\le& C\left( \Vert \nabla v^1\Vert _{L^4 (\Omega ) } + \Vert \nabla v^2\Vert _{L^4 (\Omega ) } \right) \Vert \nabla v\Vert _{L^4 (\Omega ) } 
+ C(\Vert v\Vert _{C^{0,\alpha} (\overline{\Omega   } ) } + \Vert \nabla v\Vert _{C^{0,\alpha} (\overline{\Omega } )  }    + \vert l\vert + \vert r\vert) \nonumber \\
&\le&  CT' \interleave (l,r,\omega )\interleave . \label{ccc1}
\end{eqnarray}
We infer from \eqref{bbb1} that $(\hat l , \hat r)$ satisfies
\[
\left(\begin{array}{c}
\hat l(t) \\
\hat r (t)
\end{array}\right)
=
\mathcal{J}^{-1}\int_0^t\left\{
\left(\begin{array}{c}
\Big(\int\limits_{\Omega} \nabla \mu  (\tau,y)\cdot \nabla \phi_i(y) \,d y\Big)_{i=1,2,3}
\\
\Big(\int\limits_{\Omega} \nabla \mu (\tau,y)\cdot \nabla \varphi_i(y) \,d y\Big)_{i=1,2,3}
\end{array}\right) -
\left(\begin{array}{c}
m_0 (r\times l^1 +r^2\times l) 
\\
r\times J_0r^1 +r^2\times J_0 r
\end{array}\right)
\right\}d\tau,
\]
and hence, with \eqref{ccc1}, 
\be
\label{fff3}
\vert \hat l (t) \vert + \vert \hat r(t)\vert \le  CT' \interleave (l,r,\omega )\interleave , \quad t\in [0,T']. 
\ee
Gathering together \eqref{ppp3}, \eqref{fff1}, \eqref{fff2} and \eqref{fff3}, we obtain 
 \be
\interleave (\hat l, \hat r, \hat \omega )\interleave 
 \le  CT' \interleave (l,r,\omega )\interleave , \quad t\in [0,T']. 
\ee
Thus, for $T'<1/C$, we have that 
\[
\interleave  \textnormal{\T} (l^1,r^1,\omega ^1)  - \textnormal{\T} (l^2,r^2,\omega ^2)  \interleave \le k\interleave   (l^1,r^1,\omega ^1) 
- ( l^2, r^2, \omega ^2) \interleave 
 \]
 for some constant $k\in (0,1)$, i.e. $ \textnormal{\T} $ is a contraction in $\mathcal C$. The proof of 
 Theorem  \ref{theoreme2.2} is complete. 
 \end{proof}



 \subsection{Existence of a solution of system \eqref{S1}-\eqref{S7}.}

Let us now check that the fixed-point  $(l,r,\omega)$ given in Theorem \ref{theoreme2.2} yields a solution of \eqref{S1}-\eqref{S7}. Let $v$ and $\mu$ be given by \eqref{eq_velo}-\eqref{eqEta_4} and \eqref{system_mu}, respectively. Then, since 
$(l,r,\omega ) \in {\mathcal C} \subset C([0,T'],{\mathcal F} )$, then \eqref{clas01} holds, $\nabla \mu \in C( [0,T'], L^2(\Omega))$ and hence, with 
 \eqref{sistema_l_r}, \eqref{clas03} holds as well. 
Let us set  
 \beq\label{pressure}
 {\bf q}:=\mu-\sum_{i=1}^3\dot{l}_i\phi_i-\sum_{i=1}^3\dot{r}_i\varphi_i.
 \eeq
 Then \eqref{clas02} holds and we have for a.e. $t\in (0,T')$, $q(t,.)\in C^{2, \alpha }_{loc} ( \overline{\Omega} )$  and $\lim_{ |y| \to + \infty} \nabla {\bf q} (t,y) =0$.  

\begin{proposition}
\label{proposition2.3}
Let $T'$ be as  in Theorem \ref{theoreme2.2} and let $(l,r,\omega)$  denote the corresponding fixed-point of  \T  in $\mathcal{C}$.
Then $(v, {\bf q},l,r)$ is a solution of  \eqref{S1}-\eqref{S7} in $(0,T')$. 
\end{proposition}

\begin{proof}[Proof of Proposition  \ref{proposition2.3}:] 
Let
\be
\label{aaa30}
f:=\big((v-l-r\times y)\cdot\nabla\big)v+r\times v .
\ee
Then we have that $f(t, .)\in C^{1,\alpha}_{loc} (\overline{\Omega })$ for all $t\in [0,T']$. 
On the other hand, since 
$v\in C([0,T'],M^p_{2, \delta +1})$, we have that $|v| \le C \langle y\rangle ^{-1-\delta }$ and $|\nabla v | \le C \langle y\rangle ^{-2-\delta }$, so that 
\[
| f(t, y) | \le C \langle y \rangle ^{-1-\delta }.  
\]
The divergence of $f$ is given by
\begin{eqnarray*}
\textnormal{div }f&=&\D\textrm{div\,}\Big( \big((v-l-r\times y)\cdot\nabla\big)v+r\times v\Big)\\
&=&\D\partial_i\Big(v_j\partial_jv_i\Big)-\partial_i\Big(l_j\partial_jv_i\Big)-\partial_i\Big((r\times y)_j\partial_jv_i\Big)+\textrm{div\,}\Big(r\times v\Big)\\
&=&\D(\partial_iv_j)\,(\partial_jv_i)                                                                                                                                                                                                                                                                                                           -(r\times \partial_iy)_j\partial_jv_i+\textrm{div\,}\Big(r\times v\Big)\\
&=&\D\textnormal{tr\,}\Big(\nabla v\cdot \nabla v\Big)                                                                                                                                                                                                                                                                                                           -\partial_j\Big((r\times v_i\partial_iy)_j\Big)+\textrm{div\,}\Big(r\times v\Big)\\
&=&\D\textnormal{tr\,}\Big(\nabla v\cdot \nabla v\Big)                                                                                                                                                                                                                                                                                                           -\textrm{div\,}\Big(r\times v\Big)+\textrm{div\,}\Big(r\times v\Big)\\
&=&\D\textnormal{tr\,}\Big(\nabla v\cdot \nabla v\Big)\\
&=&-\Delta \mu,
\end{eqnarray*}
where we used Einstein's convention of repeated indices and the fact that $\textnormal{div\,}(v) = 0$.
Therefore, using \eqref{eq for phi} and  \eqref{pressure}, we obtain
\beq\label{eq_f_div01}
\textnormal{div}\, (f+\nabla {\bf q})=0.
\eeq
Now we turn our attention to the curl of $f$.  Define $\tilde v :=v-l-r\times y$. Then
\beq \label{rot v tilde}
\textnormal{curl} \, \tilde{v}=\omega-2r.
\eeq
We shall use the following identities (see e.g. \cite{kikuchi86})
\ba
\label{identity 1}
\textrm{curl}(( v\cdot\nabla) v)&=&( v\cdot\nabla)\textrm{curl}( v)-(\textrm{curl}( v)\cdot\nabla) v+\textrm{div}( v)\textrm{curl}( v),\\
\label{identity 3}
\textrm{curl}(r\times v)&=&\textrm{div}(v)r-(r\cdot\nabla)v.
\ea

Applying the operator $\textrm{curl}$  to $f$ and using \eqref{identity 1}-\eqref{identity 3}, we obtain 
\begin{eqnarray*}
\textnormal{curl} \, f&=& \textrm{curl}((\tilde v\cdot\nabla)\tilde  v)+\textrm{curl}((\tilde v\cdot \nabla) (l+ r\times y)) +\textrm{curl}( r\times v)\\
&=& \textrm{curl }((\tilde v\cdot\nabla)\tilde  v)+\textrm{curl}(r\times \tilde{v}) +\textrm{curl}( r\times v)\\
&=&  \textrm{curl }((\tilde v\cdot\nabla)\tilde  v)+\textrm{div}(\tilde v)r-(r\cdot\nabla)\tilde v+\textrm{div}(v)r-(r\cdot\nabla)v\\
&=&  \textrm{curl }((\tilde v\cdot\nabla)\tilde  v)-(r\cdot\nabla)\tilde v-(r\cdot\nabla)v\\
&=& ( \tilde v\cdot\nabla)\textrm{curl}( \tilde v)-(\textrm{curl}( \tilde v)\cdot\nabla) \tilde v-(r\cdot\nabla)\tilde v-(r\cdot\nabla)v\\
&=& ( \tilde v\cdot\nabla)(\omega-2r)- ((\omega-2r)\cdot\nabla) \tilde v-(r\cdot\nabla)\tilde v-(r\cdot\nabla)v\\
&=& ( \tilde v\cdot\nabla)\omega-\omega\cdot\nabla \tilde v.
\end{eqnarray*}

Using \eqref{equation rot01}, we see that $\omega$ satisfies
\beq\label{aaa20}
\frac{\partial \omega }{\partial t}+(\tilde{v}\cdot\nabla)\omega -(\omega\cdot\nabla)\tilde{v}=0, \quad t\in (0,T').
\eeq
It follows that
 \beq\label{eq_f_rot01}
 \textnormal{curl }f+\frac{\partial \omega }{\partial t}=0.
 \eeq
 On other hand, using \eqref{system_mu} we obtain that
 \ba
 f\cdot\n &=& -\frac{\partial \mu}{\partial n}-\!\!\!\sum\limits_{1\leq j\leq m}\dot{w}_j(t)\chi_j(y)\\
 &=&-\frac{\partial q}{\partial n}-(\dot{l}+\dot{r}\times y)\cdot\n -\!\!\!\sum\limits_{1\leq j\leq m}\dot{w}_j(t)\chi_j(y).\label{eq_f_fnu}
 \ea 

Introduce now the function
\beq
\label{aaa3}
F(t,y) :=v(t,y)-v_0(y) +\int_0^t (f(s,y)+\nabla {\bf q}(s,y) )ds.
\eeq
Then $F(t,.) \in C^{ 1 , \alpha}_{loc} (\overline{\Omega })$ for all $t\in [0,T']$. On the other hand, 
it follows from \eqref{eq_f_div01}, \eqref{eq_f_rot01} and \eqref{eq_f_fnu} that
\begin{eqnarray*}
\textnormal{div }F&=&0\;\;\;\textnormal{in }\Omega,\\
\textnormal{curl }F&=&0\;\;\;\textnormal{in }\Omega,\\
F\cdot\n&=&0\;\;\;\textnormal{on }\partial\Omega , \\
\lim_{ |y| \to + \infty} F(t,y) &=& 0. 
\end{eqnarray*}
Then we infer  from \cite[Lemma 2.7]{kikuchi86} that $F\equiv 0$. Taking into account the definition of $F$, this implies that 
$v\in C^1([0,T']; C^{1, \alpha}_{loc} (\overline{\Omega}))$ with  \eqref{S1} satisfied together with $v(0,.)=v_0$. Using  \eqref{eq_velo}-\eqref{eqEta_4}, we see that the equations \eqref{S2}-\eqref{S4} 
are satisfied. Finally, the equations \eqref{S5}-\eqref{S7} hold by  \eqref{sistema_l_r} and \eqref{pressure}.   
\qed

\subsection{Proof of the estimate \eqref{aaa1}.}

The potential solution $(\overline{l},\overline{r},\overline{v})$  of \eqref{S1}-\eqref{S7} associated with $l_0, r_0, 	\{w_j\}_{1\leq j\leq m}$, and
$\overline{\omega}_0 = 0$  is obtained in the following way. Since $\overline{\omega}_0=0,$ $\pi(\overline{\omega}_0) = 0$ in $\R^3$, and hence with \eqref{equation rot01}
 the vorticity $\overline{\omega}$ is null. Then we infer from  \eqref{eqEta_1}-\eqref{eqEta_4} 
 that $\overline{\eta} =0$ and from \eqref{eq_velo}
that
\beq \label{eq_veloPotential}
\overline{v}=\sum\limits_{i=1}^3\overline l_i\nabla \phi_i+\sum\limits_{i=1}^3\overline r_i\nabla \varphi_i+\sum\limits_{1\leq j\leq m}w_j(t)\nabla\psi_j.
\eeq

It follows from \cite[Proposition 2.3]{LR} that $(\overline l, \overline r)$ satisfies the ODE  \eqref{system lr}, whose solution is unique. 

Consider a solution $(l,r, \omega )$ associated with the same $l_0,r_0, \{w_j\}_{1\leq j\leq m}$ as for $(\overline{l}, \overline{r}, \overline{\omega})$, but with 
an initial vorticity $\omega _0$ not necessarily null.  It follows from \eqref{a1}-\eqref{a2} that for all $t\in [0,T']$
\[
\Vert \omega (t) \Vert _{C^{1, \alpha} (\overline{\Omega})  }
+ \Vert \omega (t) \Vert _{M^p_{1, \delta + 2}} \le P = e^e \cdot  ( C_6  \| \pi( \omega_0 ) \|_{C^{1,\alpha}(\R^3)} + C_7 \| \pi ( \omega _0 )\| _{  M^p_{1, \delta +2 } (\R ^3 )  } ).
\]
Now, from \eqref{eq_velo} and \eqref{eqEta_1}-\eqref{eqEta_4}, we infer that for all $t\in [0,T']$
\[
\Vert v(t) -\overline v(t) \Vert _{C^{2,\alpha} (\overline{\Omega})} + \Vert \nabla v(t) -\nabla \overline v(t)\Vert _{L^4(\Omega )}
\le C \left( P + \vert (l(t)-\overline l(t), r(t) -\overline r(t) )  \vert \right) . 
\]
Combined with \eqref{system_mu}, this yields
\[
\Vert \nabla \mu (t)  -\nabla \overline \mu (t)  \Vert _{L^2(\Omega )}
\le C \left( P + \vert ( l(t)-\overline l(t), r(t) -\overline r(t))\vert \right) . 
\]
Using \eqref{sistema_l_r}, we obtain 
\[
\vert ( \dot l (t) -\dot { \overline{l}} (t), 
\dot r(t) -\dot {\overline{r}} (t) )  \vert
\le C \left( P + \vert (l(t)-\overline l(t), r(t) -\overline r(t))\vert \right). 
\]
Then \eqref{aaa1} follows by using Gronwall's lemma.  The proof of Theorem \ref{theoreme1} is complete. \end{proof}

\subsection {Uniqueness and continuity with respect to the control}

The following result is concerned with the uniqueness of the solution $(l,r,v,{\bf q})$ of \eqref{S1}-\eqref{S7}, when the vorticity 
$\omega = \textrm{curl } v$ satisfies
\be
\label{rr1}
\omega (t,y) = G^{-1} (0; t,y) \pi (\omega _0)( \hat X (0; t, y))
\ee
where $G(s;t,y)= (\partial \hat X/\partial y)(s;t,y)$ and the flow $\hat X$ is defined on $[0,T']^2 \times \R^3$ by 
\be
\label{rr2}
\left\{
\begin{array}{ccl}
\D\frac{\partial}{\partial s} \hat{X}(s;t,y)&=&  \pi (v) (s,\hat{X}(s;t,y))-l(s)-r(s)\times\hat{X}(s;t,y),\\\\
\hat{X}(t;t,y)&=& y.
\end{array}\right.
\ee

\begin{proposition} Let $l_0,r_0,\omega _0,v_0$ and $T'$ be as in Theorem \ref{theoreme1}. Then the solution  $(v,{\bf q},l,r,\omega )$ of \eqref{S1}-\eqref{S7} and 
\eqref{rr1}-\eqref{rr2} is unique in the class \eqref{clas01}-\eqref{clas03} ($\bf q$ being unique up to the addition of an arbitrary function of time). 
On the other hand, for any given initial data $(l_0,r_0,\omega _0)$ as above and any $R>0$, 
the map $w\in {\mathcal B} :=\{ w\in C^1( [0,T'] ,\R ^m); \ \Vert w\Vert _{ C^1( [0,T] ) } \le R \}  \mapsto (l,r)\in C^0 ([0,T'], \R^6)$ is continuous.  
\label{prop2}

\end{proposition}
\begin{proof}
Let  $(v,{\bf q},l,r)$ be a solution of \eqref{S1}-\eqref{S7} in the class 
\eqref{clas01}-\eqref{clas03}. Then we can expand $v$ in the form \eqref{eq_velo} with $\eta$ as in \eqref{eqEta_1}-\eqref{eqEta_4}. 
Then it is well-known that the vorticity $\omega = \textrm{curl } v$ satisfies the equation \eqref{aaa20} with $\tilde v = v - l - r\times y$, and that it is given by 
\eqref{rr1} ``away'' from the rigid body. We assume that it is given by \eqref{rr1} everywhere, even on $\partial \Omega$. Roughly speaking, this amounts
to specifying the tangent components of the vorticity on the inflow section.  
Let us show that  the pair $(l,r)$ satisfies \eqref{sistema_l_r}. Let $\mu$ be as in \eqref{pressure} and let $f$  be as in \eqref{aaa30}. 
Then by \eqref{S1} and the computations above, we have that 
\[
-\Delta \mu = - \Delta q = \textrm{ div }f = \textrm{tr } (\nabla v\cdot \nabla v),  
\]
and 
\begin{eqnarray*}
\frac{\partial \mu}{\partial n} &=& \frac{\partial q}{\partial n } + \sum_{i=1}^3 \dot l_i n_i + \sum_{i=1}^3 \dot r_i (y\times n)_i \\
&=& -(\frac{\partial v}{\partial t} + f )\cdot n + \dot l \cdot n + \dot r \cdot (y\times n) \\
&=& -\frac{\partial }{\partial t} \left(  [ l+r\times y]\cdot n  +\sum_{1\le j\le m} w_j(t) \chi _j(y) \right)  - f\cdot n  + \dot l \cdot n + \dot r \cdot (y\times n) \\
&=& -\sum_{1\le j\le m} \dot w_j (t) \chi _j (y) - \big( (v-l-r\times y)\cdot \nabla v + r\times v\big) \cdot n .
\end{eqnarray*}
Thus $\mu$ solves \eqref{system_mu}. Integrating in \eqref{S5}-\eqref{S6} and using \eqref{pressure}, we arrive to \eqref{sistema_l_r} with 
$(\hat l, \hat r) = (l,r)$. Thus $(l,r,\omega)$ is a fixed-point of \T. As there is (for $T'$ small enough) {\em 
only one} fixed-point of \T by the contraction mapping 
theorem, we infer that $(l,r,\omega )$ is unique.  Then $\eta$ is unique by   \eqref{eqEta_1}-\eqref{eqEta_4}, and $v$ is unique by \eqref{eq_velo}. Finally, 
$\nabla q$ is unique by \eqref{S1} and $q$ is unique (up to the addition of an arbitrary function of time). 

Let us proceed with the continuity with respect to the control.  Assume given some initial data $(l_0,r_0,\omega _0)$ as above and pick any number $R>0$. 
Let 
\[
{\mathcal B} :=\{ w\in C^1  ( [0,T], \R ^m); \Vert w\Vert _{C^1( [0,T] )} \le R\} .
\] 
Assume that the constants $C_8$ and $C_{11}$ are 
suitably chosen to be convenient for all $w\in {\mathcal B}$, and pick a time $T'>0$ convenient for all $w\in {\mathcal B}$. 
Then\\
(i) for $T'$ small enough, we have for $w\in {\mathcal B}$ and $(l^i,r^i,\omega ^i)\in {\mathcal C}$, $i=1,2$,
\[
\interleave  \textnormal{\T} (l^1,r^1,\omega ^1)  - \textnormal{\T} (l^2,r^2,\omega ^2)  \interleave \le k\interleave   (l^1,r^1,\omega ^1) 
- ( l^2, r^2, \omega ^2) \interleave , \quad 
 \]
 for some constant $k\in (0,1)$;\\
(ii) for given $(l,r,\omega)\in {\mathcal C}$, the map $w\in {\mathcal B} \mapsto (\hat l, \hat r, \hat \omega ) \in {\mathcal C} $ is continuous. 

Indeed, the map 
$w\in {\mathcal B} \mapsto v \in C ( [0,T'], C^{2, \alpha} ( \overline{\Omega}) \cap M^p_{2, \delta + 1} )$ is clearly continuous (using 
\eqref{eq_velo} and \eqref{eqEta_1}-\eqref{eqEta_4}), and hence the map $w\in {\mathcal B} \mapsto (l,r) \in C^1([0,T'], \R ^6)$ is continuous
(by \eqref{system_mu}-\eqref{sistema_l_r}). Finally, using the assumption $\omega _0\in M^p_{0, \delta  +3}$,  \eqref{a1}, \eqref{aaa20}, 
Aubin-Lions'  lemma and the continuity of $v$, one can see (as e.g. in \cite{GR}) that the 
map $w\in {\mathcal B} \mapsto \omega \in C([0,T'],C^{0,\alpha} (\overline{\Omega}) \cap L^p _{p(\delta + 2)} (\Omega))$ is continuous.    
 
It follows again from the contraction mapping theorem (for a map depending on a parameter) that the map which associates with 
$w\in {\mathcal B}$ the fixed-point $(l,r,\omega)\in {\mathcal C}$ is continuous. 
 \end{proof}

\section{Proof of main result}

We are now in a position to prove the main result in this paper. Let $T_0,P,N,K$ and $R$ be some
given positive numbers. Then by Theorem \ref{theoreme1}, there exists a time 
$T=T(T_0, P,N,K,R) \in (0,T_0]$ such that system \eqref{S1}-\eqref{S7} has a solution $(v,{\bf q}, l,r)$ 
for $t\in [0,T]$, with $(l,r,\omega )\in C([0,T],{\mathcal F}) $, provided that  $ \vert ( l_0,r_0) \vert \le 1 $, $\Vert w\Vert _{C^1([0,T_0]  )} \le R$ and 
 $\omega_0$ satisfies   \eqref{Homega}-\eqref{zzz2} and 
 \[
 \Vert \omega _0\Vert _{ C^{1,\alpha} (\overline{\Omega} ) } + \Vert \omega _0\Vert _{ M^p_{1, \delta +2} }
  \le P, \quad \textrm{div } \omega _0=0, \qquad \int_{\partial \Omega} \omega_0\cdot n \, d \sigma =0. 
 \]
Let $\Pi$ be a (continuous and linear) 
extension operator from $C^1([0,T])$ to $C^1([0,T_0])$ and pick $\delta :=R/ \Vert \Pi \Vert $.  
Then $\Vert \Pi (w)\Vert _{C^1 ( [0,T_0])}\le R$ if  $\Vert w\Vert _{C^1([0,T])} \le \delta $. In particular, 
using assumption (H) for the time $T$, we have  that $\Vert \Pi (w) \Vert _{ C^1( [0,T_0] )  } \le R$
if $w=W(h_0, \vec{q_0}, l_0,r_0, h_T, \vec{q_T}, l_T,r_T)$ for 
$\vert (h_0, \vec{q_0}, l_0,r_0, h_T, \vec{q_T}, l_T,r_T) \vert < \eta$
with $\eta >0$ small enough.  Then system \eqref{S1}-\eqref{S7} has a solution defined for $t\in [0,T]$ corresponding
to $(l_0,r_0,\omega _0, w)$ as above, and also a potential solution corresponding to the same data $(l_0,r_0,w)$ and
to $\bar{\omega} _0 \equiv 0$.

Let $\omega _0$ be as in the statement of Theorem \ref{thm1},  and write
$a_0=(h_0,\vec{q}_0)$, $b_0=(l_0,r_0)$, $a_T=(h_T,\vec{q}_T)$, and $b_T=(l_T,r_T)$.
Let $a(t):=( h(t),\vec{q}(t) )$ and $b(t)=(l (t), r(t))$. 
The proof is done in two steps. In the first step, we prove the result for 
$||\omega _0||_{C^{1,\alpha} (\overline{\Omega}) }$, $||\omega _0||_{M^p_{1, \delta + 2}}$,  $| l_{0}|$, $|r_{0}|$, $|l_{T}|$ and $|r_{T}|$ small enough, and in the second step, we remove
this assumption by performing a scaling in time.\\[2mm]

%
\noindent
{\sc Step 1.} Let the map $W$ be as in the assumption (H) for the time $T$.
We may pick a number  $\eta_{1} \in (0,1)$ such that
$w=W(a_0,b_0,a_T,b_T)$ is defined for $|(a_0,b_0)| \le  \eta _1$ and  $|(a_T,b_T)| \le  \eta _1$, with
\[ 
||w||_{C^1([0,T])}\le \delta  .
\] 
Pick any initial state $(a_0,b_0)=(h_0,\vec{q}_0,l_0,r_0)$ with 
$|(a_0,b_0)| \le \eta _1$. For any given $(a_T,b_T,v_0)$ with $|(a_T,b_T)| \le \eta _1$,  we denote by
$(h,\vec{q}, l, r,v,{\mathbf q})$ the solution of \eqref{S1}-\eqref{S7} and \eqref{systempq} corresponding to
the velocity $v_0$ and to the control 
$w=W(a_0,b_0,a_T,b_T)$, and  by
$(\overline{h},\overline{\vec{q}},\overline{l},\overline{r}, \overline{v}, \overline{\mathbf q})$ 
the solution corresponding to  $(a_0,b_0)$ together with  the velocity $\overline{v}_0$ which solves
\begin{eqnarray*}
\text{\rm curl } \overline{v}_0 &=& 0,\qquad \textrm{ in } \Omega ,\\
\text{\rm div } \overline{v}_0  &=& 0,\qquad \textrm{ in } \Omega , \\
\overline{v}_0\cdot n  &=& (l_0+r_0\times y)\cdot n  ,\quad \textrm { on } \partial \Omega , \\
\lim_{|y|\to \infty } \overline{v}_0 (y) &=&0,
\end{eqnarray*}
and to the (same) control $w$. 
From \eqref{aaa1} we infer 
that there exists some constant  $C_1>0$ such that 
\begin{equation}
\label{WWW}
||(l-\overline{l} , r-\overline{r})||_{L^\infty (0,T)} 
\le C_1 \Big(\norm{\omega_0}_{C^{1,\alpha}(\overline{\Omega})}+\norm{\omega_0}_{M^p_{1,\delta + 2 } }\Big),
\eeq
whenever 
\begin{equation}
\label{condition}
|(l_0,r_0)|\le 1,\quad 
||\omega _0||_{C^{1,\alpha } ( \overline{\Omega} ) } + ||\omega  _0||_{ M^p_{1, \delta + 2 }}  \le P,\quad 
\text{ and } \ ||w||_{ C^1( [0,T] ) }\le \delta .
\end{equation}
Combined to the equations
\[
\left\{\begin{array}{ccl}
h'(t)&=&  (1- |  \vec{q} | ^2) l+ 2\sqrt{1- | \vec{q} | ^2}\,   \vec{q}\times l 
 + (l\cdot  \vec{q}\, ) \vec{q} -  \vec{q} \times l\times  \vec{q}, \\[3mm]
{\vec{q}\, }'(t)&=& \frac{1}{2} (\sqrt{1- | \vec{q} |^2} \, r +  \vec{q} \times r ), \\[3mm]
{\overline h}'(t)&=&  (1-  |  \overline{\vec{q}} | ^2) \overline{l} + 
2\sqrt{1- | \overline{\vec{q}} | ^2}\,  \overline{ \vec{q} }\times \overline{ l}  + 
(\overline{l}\cdot  \overline{\vec{q}}\, )
\overline{\vec{q}} -  \overline{\vec{q}} \times \overline{l}\times  \overline{\vec{q}}, \\[3mm]
{ \overline{\vec{q}}\, }'(t)&=& \frac{1}{2} (\sqrt{1- | \overline{\vec{q}} |^2} \, \overline{r} +  \overline{\vec{q}} \times \overline{r} ), \\[3mm]
h(0)&=&\overline{h} (0)=h_0, \quad \vec{q}(0)=\overline{\vec{q}}(0)=\vec{q}_0,\\[3mm] 
\end{array}\right. 
\]
this gives for some constant $C_2>0$
\begin{equation}
\label{AvecOuSansBis}
||(h-\overline{h}, \vec{q} -\overline{\vec{q}})||_{L^\infty (0,T)}
\le C_2 \left( ||\omega _0||_{C^{1,\alpha } ( \overline{\Omega} ) } + ||\omega  _0||_{ M^p_{1, \delta + 2} } \right),
\end{equation}
provided that \eqref{condition} holds.
Let  $f:\overline{B} =\{x\in \R ^{12};\ |x|\le 1\}\to \R^{12}$ be defined by
\[
f(x_T)= \frac{1}{\eta_{1}} (a(T),b(T))=\frac{1}{\eta _1} (h(T),\vec{q}(T), l(T), r(T))
\]
where $(a_T,b_T)=: \eta _1 x_T$. \par
We notice that $f$ is continuous, by virtue of Proposition \ref{prop2} and \eqref{systempq}.
Pick any $\varepsilon \in (0,1)$. From \eqref{WWW} and \eqref{AvecOuSansBis}, we deduce that for
\begin{eqnarray}
&&|| \omega_0 ||_{C^{1,\alpha}(\overline{\Omega} )}
+ || \omega _0  ||_{M^p_{1, \delta + 2 }}
<\nu 
\label{condvor}
\end{eqnarray}
with $\nu  >0$ small enough, we have that 
\[
|f(x_T)-x_T|<\varepsilon,\qquad \text{ for } |x_T|\le 1.
\]
We need the following topological result \cite[Lemma 4.1]{GR}.
\begin{lemma}
\label{topo}
Let $B=\{ x\in \R ^n;\ |x|<1\}$ and $S=\partial B$. Let $f:\overline{B}\to \R ^n$ 
be a continuous map such that for some constant $\varepsilon \in (0,1)$
\begin{equation}
|f(x)-x| \le \varepsilon\qquad \forall x\in S.
\label{deg1} 
\end{equation}
Then 
\begin{equation}
\label{deg2}
(1-\varepsilon) B \subset f(\overline{B}).
\end{equation}
\end{lemma}
\noindent
\ \par
Thus, we infer from Lemma \ref{topo} that if $(a_0,b_0,a_T,b_T)\in \R ^{24}$ is such that
\[
|(a_0,b_0)| <\eta _1,\quad
|(a_T,b_T)|< \eta_{2} := \eta_{1} (1-\varepsilon ),
\]
and \eqref{condvor}  is satisfied,  
then there exists a control $w=W(a_0,b_0,\tilde a_T,\tilde b_T)$ for which the solution of \eqref{S1}-\eqref{S7} and \eqref{systempq} satisfies
$(h(T),q(T),l(T), r(T))= (a(T),b(T)) =  (a_T,b_T)$.\\[3mm]
\noindent
{\sc Step 2.} To drop the assumptions $ |b_0| <\eta _1$, $ |b_T | < \eta _2$, and \eqref{condvor} (corresponding to a 
given time $T\in (0,T_0]$), we 
use a  scaling in time introduced in \cite{coron1} for the control of
Euler equations. Let  $(a_0,b_0)$, $(a_T,b_T)$, and $v_0$ be given data with 
\begin{equation*}
|a_0|<\eta_{2} \ \textrm{ and }  \  |a_T|< \eta_{2}.
\end{equation*}
We set $b_0^\lambda := \lambda b_0$, $b_T^\lambda := \lambda b_T$, and $v_0^\lambda :=\lambda v_0$.
Then for $\lambda >0$ small enough, we have that 
\[
|(a_0,b_0^\lambda )|<\eta_{2}, \quad |(a_T,b_T^\lambda )|< \eta_{2},
\]
and  $\omega_0^\lambda :=\text{curl } v_0^\lambda$ satisfies
\[
|| \omega_0^\lambda  ||_{C^{1,\alpha}(\overline{\Omega} )}
+ || \omega _0^\lambda   ||_{M^p_{1, \delta + 2}}
<\nu .
\]
By Step 1, there exists some trajectory  $(a^\lambda, b^\lambda)$ for the underwater vehicle connecting $(a_0,b_0^\lambda )$ at $t=0$ to $(a_T,b_T^\lambda )$ 
at $t=T$, with corresponding fluid velocity $v^\lambda$, pressure ${\mathbf q}^\lambda$,  and control $w^\lambda$. Let us set
\begin{eqnarray*}
a(t)& :=&a^\lambda (\lambda^{-1}t),\\
b(t)&: =&\lambda^{-1}  b^\lambda (\lambda^{-1}t),\\
v(t,y)&: =&\lambda ^{-1} v^\lambda (\lambda ^{-1}t,y),\\
{\mathbf q} (t,y) &: =& \lambda ^{-2} {\mathbf q}^\lambda (\lambda ^{-1}t, y),\\ 
w(t)  &: =&\lambda ^{-1} w^\lambda (\lambda ^{-1}t),
\end{eqnarray*}
for $y\in \Omega$ and $0\le t\le T _\lambda :=\lambda T \in (0,T_0]$. Then  $(a,b)$ is a trajectory for the underwater vehicle connecting $(a_0,b_0)$ at $t=0$ to $(a_T,b_T)$ at 
$t=T_\lambda $ and corresponding to the initial fluid velocity $v_0$. \qed
%
%
%


\section{Acknowledgements}
The first author (RL) was partially supported by Basal-CMM project, PFB 03. The second author
(LR) was partially supported by the Agence Nationale de la Recherche,  project ANR 15 CE23 0007 01
(Finite4SoS). 

\bibliographystyle{abbrv}
\bibliography{control}
\end{document}